\documentclass[12pt]{article}

\usepackage{amsmath,amsthm,amssymb,amsfonts,stmaryrd}
\usepackage{latexsym,graphicx,color}

\newtheorem{remark}{Remark}[section]
\newtheorem{lemma}{Lemma}[section]
\newtheorem{theorem}{Theorem}[section]
\newtheorem{prop}{Proposition}[section]
\newcommand{\set}[1]{\left\{#1\right\}}
\newcommand{\dual}[1]{\langle#1\rangle}

\newcommand{\jump}[1]{\llbracket #1 \rrbracket}
\newcommand{\mean}[1]{\{#1\}}

\newcommand{\abs}[1]{\lvert#1\rvert}
\newcommand{\norm}[1]{\lVert#1\rVert}

\newcommand{\nder}[1]{\disp\frac{\partial#1}{\partial \n}}

\newcommand{\n}{\boldsymbol{n}}
\newcommand{\tg}{\boldsymbol{t}}
\newcommand{\x}{\boldsymbol{x}}
\newcommand{\y}{\boldsymbol{y}}
\newcommand{\curl}{\mathbf{curl}}

\newcommand{\cT}{\mathcal{T}}
\newcommand{\cF}{\mathcal{F}}
\newcommand{\cE}{\mathcal{E}}
\newcommand{\cG}{\mathcal{G}}
\newcommand{\cP}{\mathcal{P}}

\newcommand{\bsigma}{\boldsymbol{\sigma}}
\newcommand{\btau}{\boldsymbol{\tau}}
\newcommand{\bSigma}{\boldsymbol{\Sigma}}
\newcommand{\bPi}{\boldsymbol{\Pi}}
\newcommand{\bbeta}{\boldsymbol{\beta}}

\newcommand{\disp}{\displaystyle}

\newcommand{\pw}{\mathfrak{b}}

\date{}
\title{Symmetric coupling of LDG-FEM and DG-BEM
\thanks{Partial support by the following institutions is gratefully acknowledged:
CONICYT through projects Anillo ACT1118 (ANANUM) and Fondecyt 1110324,
Spain's Ministry of Education through Project MTM2010-18427, and
NSF through grant DMS 1216356.}}

\author{{\sc Norbert Heuer}\thanks{
Facultad de Matem\'aticas, Pontificia Universidad Cat\'olica de Chile, 
Avenida Vicu\~na Mackenna 4860, Santiago, Chile,
e-mail: {\tt nheuer@mat.puc.cl}}, $\,\,$
{\sc Salim Meddahi}\thanks{Departamento de Matem\'aticas, Facultad de Ciencias,
Universidad de Oviedo, Calvo Sotelo s/n, Oviedo, Espa\~na,
e-mail: {\tt salim@uniovi.es}}\\ and \\
{\sc Francisco-Javier Sayas}\thanks{Department of Mathematical Sciences,
University of Delaware, Newark DE 19716, USA, e-mail: {\tt
fjsayas@math.udel.edu}}
}

\begin{document}

\maketitle

\begin{abstract}
 We analyze a discontinuous Galerkin FEM-BEM scheme for a second order elliptic  
 transmission problem posed in the three-dimensional space.
 The symmetric variational formulation is discretized by
 nonconforming Raviart-Thomas finite elements 
 on a general partition of the interior domain coupled with discontinuous boundary elements
 on an independent quasi-uniform mesh of the transmission interface. 
 We prove (almost) quasi-optimal convergence of the method 
 and confirm the theory by a numerical experiment.
 In addition we consider the case when continuous rather than
 discontinuous boundary elements are used.
\end{abstract}

\section{Introduction}

Discontinuous Galerkin (DG) methods are known to be  flexible and efficient solvers for      
a wide range of partial differential equations. Among their advantages, when applied to 
second order elliptic problems, we emphasize that  
they  are locally conservative, they can handle general meshes with hanging nodes and they 
allow the use of different polynomial degrees in each element.

DG methods can be coupled with the boundary element method (BEM) in different ways \cite{CockburnSayas}.
In \cite{GaticaSayas} it was shown that it is possible to benefit from the features
highlighted above when approximating non-homogeneous
(and even nonlinear \cite{GaticaHeuerSayas}) exterior 
elliptic problems  if a local discontinuous Galerkin method (LDG) is used as  
an interior solver in combination with the BEM.  

The symmetric LDG-BEM formulation is obtained by rewriting locally the elliptic problem in mixed form
and considering a Calder\'on identity on the boundary. In this way, one ends up with  a system of 
two variational equations in the interior domain (involving both the potential and the flux 
as independent variables) and a system of two boundary integral equations  relating the Cauchy datum 
of the problem on the coupling interface. In the first LDG-BEM formulation \cite{GaticaSayas}, 
the coupling between the two systems 
is performed by using Costable's approach. From the DG point of view, this amounts to using the 
normal derivative of the solution on the coupling boundary 
as a Neumann datum when defining the numerical fluxes 
for the LDG method. In the resulting coupled scheme, the normal derivative becomes an 
independent unknown and the other BEM variable (the discrete trace)    
must match the discrete potential that comes 
from the LDG method. The problem is that these unknowns are of different nature: the restriction of the LDG 
approximation of the potential to the coupling interface is discontinuous while the BEM discretization  
is conforming and produces a continuous and piecewise polynomial approximation of this variable. 
This inconvenience is addressed in \cite{GaticaSayas} by introducing a further unknown that acts as 
a Lagrange multiplier and enforces weakly the imposition of the missing transmission condition. 
A later paper \cite{GaticaHeuerSayas} eliminated the need of the Lagrange multiplier by demanding that the 
discontinuous piecewise polynomial functions that approximate the potential in the LDG method be continuous 
at the coupling interface. Here, the normal derivative is the only boundary unknown, which 
reduces the number of unknown functions by two with respect to the first version. However, 
in order to deal properly with this formulation in practice
a Lagrange multiplier must come again into play. Moreover, 
this formulation imposes for the BEM the mesh inherited from the interior partition of the domain, which 
reduces much of the flexibility provided by the discrete Galerkin method near the coupling boundary. 
Finally, we point out that recently non-symmetric couplings of DG with BEM have also been studied, cf.  
\cite{HeuerSayas} and the references therein.

In this paper, following \cite{salim},  
we take advantage of the fact that the flux variable is an LDG active unknown (as in the traditional mixed formulation) 
and consider a dual approach: we define the numerical fluxes by considering the trace of the solution on the coupling 
boundary as Dirichlet datum. Hence, as opposed to the former strategy, the trace of the solution is an independent 
variable while the LDG normal flux and the normal derivative must be merged on the coupling boundary. Notice that 
in this case both variables are (naturally) nonconforming and no Lagrange multiplier or special restriction 
is needed to match them. Consequently, the resulting numerical scheme enjoys all the good properties of a typical 
DG method and allows for using an independent boundary mesh. Moreover, one can employ both
a conforming or a nonconforming approximation on the boundary. 

In this paper, we  take advantage of the results from \cite{NorbertSalim} 
to deal with a DG finite element method on the boundary, the resulting scheme will be referred to 
as the LDG-FEM/DG-BEM method. To our knowledge, this is the first FEM-BEM scheme that 
combines DG approximations on the boundary and in the interior. 
Technical difficulties that already arose in \cite{NorbertSalim}, oblige 
us to consider conforming and quasi-uniform families of triangulations 
on the coupling boundary. Following the technique from \cite{ChoulyH_12_NDD}, this can be relaxed 
to meshes that are conforming and quasi-uniform on planar sub-surfaces of the coupling interface. 
However, for simplicity, the technical details for such an extension are omitted here and we will consider 
globally conforming and quasi-uniform boundary meshes.
Fortunately, restrictions on the boundary mesh have no negative impact on the triangulation of the interior 
domain since the two meshes are related by a mild local condition, see \eqref{localUnif} below. 
Finally, we analyze the scheme that is obtained by using a conforming rather than non-conforming BEM on the 
interface. The resulting scheme will be referred to as the LDG-FEM/BEM method.

The paper is organized as follows. In Section \ref{s2} we present our model problem
and recall some basic properties of boundary integral operators.
For simplicity of exposition we will restrict our interest to a three-dimensional problem posed in the whole space. 
 In Section \ref{s3} we derive the LDG-FEM/DG-BEM scheme and prove that it admits a unique solution. 
Stability and a priori error estimates are proved in Section \ref{s4}. 
In Section \ref{s5} we show that the same technical arguments provide 
(without the quasi-uniformity requirement for the meshes on the coupling boundary) a convergence result 
for the LDG-FEM/BEM scheme. 
Finally, numerical experiments are reported in Section \ref{s6}.

%%%%%%%%%%%%%%%%%%%%%%%%%%%

Given a real number $r\geq 0$ and a polyhedron $\mathcal O\subset \mathbb R^d$, $(d=2,3)$,  
we denote the norms and seminorms of the usual Sobolev space 
$H^r(\mathcal O)$ by $\|\cdot \|_{r,\mathcal O}$ and $|\cdot|_{r,\mathcal O}$ respectively   
(cf. \cite{McLean}). We use the convention $L^2(\mathcal O):= H^0(\mathcal O)$ 
and let $(\cdot,\cdot)_{\mathcal O}$ be the 
inner product in $L^2(\mathcal{O})$. 
We recall that, for any $t \in [-1,\: 1 ]$, the spaces $H^{t}(\partial \mathcal O)$  
have an intrinsic definition (by localization) on the Lipschitz surface $\partial \mathcal O$ 
due to their invariance under Lipschitz coordinate transformations. Moreover, for all $0< t\leq 1$, 
$H^{-t}(\partial\mathcal O)$ is the dual of $H^{t}(\partial\mathcal O)$ with respect
to the pivot space $L^2(\partial \mathcal{O})$.
Also, $\dual{\cdot, \cdot}_{\partial \mathcal{O}}$ denotes both the $L^2(\partial \mathcal{O})$ inner product 
and its extension to the duality
pairing of $H^{-t}(\partial \mathcal{O})\times H^{t}(\partial \mathcal{O})$.

% Throughout this paper $C$and $c$, with or without subscripts,  will denote positive constants, not
% necessarily the same at different occurrences, which are
% independent of the parameter $h$ and functions involved.

\section{The model problem}\label{s2}
Let $\Omega\subset \mathbb{R}^3$ be a bounded polyhedral domain with a Lipschitz boundary $\Gamma$. 
We denote by $\mathbf{n}$ the unit normal vector on $\Gamma$ that points towards 
$\Omega^e:= \mathbb{R}^3\setminus \bar \Omega$.
For the sake of simplicity, we assume that $\Omega^e$ is connected.
We consider the transmission problem
\begin{equation}\label{ModelProblem}
\begin{array}{rcll}
-\Delta u &=& f         &\text{in $\Omega$}\\[2ex]
        u &=& u^e + g_0 &\text{on $\Gamma$}\\[2ex]
        \nder{u} &=& \nder{u^e} + g_1 &\text{on $\Gamma$}\\[2ex]
        - \Delta u^e &=& 0 &\text{in $\Omega^e$}\\[2ex]
        u^e &=& O(\disp\frac{1}{\abs{\x}})  &\text{as $\abs{\x}\to \infty$},
\end{array}
\end{equation}
where $f\in L^2(\Omega)$, $g_0\in H^{1/2}(\Gamma)$ and $g_1\in L^2(\Gamma)$ are given functions.

We can write the problem in $\Omega$ by introducing the flux $\bsigma$ as a new variable:
\begin{equation*}\label{inOmega}
\begin{array}{rcll}
\bsigma &=& \nabla u &\text{in $\Omega$},\\[2ex]
-\textrm{div} \bsigma &=& f  &\text{in $\Omega$}.
\end{array}
\end{equation*}
With the notation 
\[
\lambda := \nder{u^e}\quad \text{and} \quad \psi = u^e|_\Gamma,
\]
the transmission conditions are 
\begin{equation}\label{trans}
\begin{array}{rcll}
u &=& \psi + g_0 &\text{on $\Gamma$}\\[2ex]
\bsigma \cdot \n &=& \lambda + g_1  &\text{on $\Gamma$}.
\end{array}
\end{equation}
Using the integral representation of the harmonic function  $u^e$ in $\Omega^e$ gives
\begin{equation*}\label{IntegralRep}
u^e = \Psi_{DL}(\psi) - \Psi_{SL}(\lambda)  \quad\text{in $\Omega^e$}
\end{equation*}
where
\[
\Psi_{SL}(\xi)(\x) := \int_{\Gamma} E(\abs{\x-\y}) \xi(\y)\, dS(\y)
\quad \text{and} \quad 
\Psi_{DL}(\varphi)(x) := \int_{\Gamma} \disp\frac{\partial E(\abs{\x-\y})}{\partial \n(\y)}\, \varphi(\y)\, dS(\y) 
\]
are the single and double layer potentials, respectively, and $E(\abs{\x}):= \frac{1}{4\pi}\frac{ 1}{\abs{\x}}$ 
is the fundamental solution of the Laplace operator. The jump properties of the single and double layer potentials 
across $\Gamma$ provide the following integral equations relating the Cauchy data on this boundary:
\begin{align}
   \label{inteq1}
   \psi &= ( \frac{\text{id}}{2} + K) \psi - V \lambda \\
   \label{inteq2}
   \lambda &= -W\psi + (\frac{\text{id}}{2} - K') \lambda
\end{align}
where $V$, $K$, $K'$ are the boundary integral operators representing the single, double and adjoint of the 
double layer, respectively, and $W$ is the hypersingular operator.

Let us recall some important properties of the boundary integral operators, see \cite{McLean} for details.
The boundary integral operators are formally defined at almost every point $\x \in \Gamma$ by
\[
 V\xi (\x):= \int_\Gamma E(\abs{\x-\y}) \xi(\y)\, dS(\y), \qquad K\varphi (\x):= \int_\Gamma 
 \disp\frac{\partial E(\abs{\x-\y})}{\partial \n(\y)}\, \varphi(\y)\, dS(\y),
\]
\[
 K'\xi (\x):= \int_\Gamma 
 \disp\frac{\partial E(\abs{\x-\y})}{\partial \n(\x)}\, \xi(\y)\, dS(\y), \quad 
 W \varphi (\x) := - \disp\frac{\partial}{\partial \n(\x)} \int_\Gamma 
 \disp\frac{\partial E(\abs{\x-\y})}{\partial \n(\y)}\, \varphi(\y)\, dS(\y).
\]
They are  bounded as mappings $V:\, H^{-1/2}(\Gamma)\to H^{1/2}(\Gamma)$, 
$K:\, H^{1/2}(\Gamma) \to H^{1/2}(\Gamma)$ and $W:\, H^{1/2}\to H^{-1/2}(\Gamma)$. 
\medskip
The single layer operator is coercive, 
there exists $C_0>0$ such that
\begin{equation}\label{eq-coercive-1}
 \dual{\chi, V\chi}_\Gamma \,\ge\, C_0\, \norm{\chi}_{-1/2,\Gamma}^2\quad 
 \forall\, \chi \,\in\,   H^{-1/2}(\Gamma)
\end{equation}
and
 \begin{equation}\label{eq-coercive-2}
\dual{W \varphi, \varphi}_\Gamma \,+\, \left(\int_\Gamma \varphi \, \right)^2 
\,\ge\, C_0\,\norm{\varphi}_{1/2,\Gamma}^2\quad 
 \forall\,\varphi \,\in\, H^{1/2}(\Gamma).
 \end{equation} 
Moreover,  $V:\, H^{s-1}(\Gamma)\to H^s(\Gamma)$ is bounded for any $0\leq s \leq 1$. 
% and
%  \begin{equation}\label{eq-coercive-2}
% \dual{W \varphi, \varphi}_\Gamma \,+\, \left(\int_\Gamma \varphi \, \right)^2 
% \,\ge\, C_0\,\norm{\varphi}_{1/2,\Gamma}^2\quad 
%  \forall\,\varphi \,\in\, H^{1/2}(\Gamma).
%  \end{equation} 
We recall that the operators $V$ and $W$ are related by
\[
 W = \textrm{curl}_\Gamma V \curl_\Gamma
\]
where $\curl_\Gamma$ is the surface curl operator and $\textrm{curl}_\Gamma$ is its adjoint operator,
cf. \cite{Nedelec_82_IEN}.
Consequently, 
\[
 \dual{W\psi, \varphi}_\Gamma = \dual{\curl_\Gamma \psi, V \curl_\Gamma \varphi}_\Gamma, \qquad 
 \forall \psi, \varphi \in H^{1/2}(\Gamma).
\]
% Here,  
% In the following, $\curl_T$ will denote the restriction of $\curl_\Gamma$ onto $F\subset \Gamma$.

\section{The LDG-FEM/DG-BEM formulation}\label{s3}
We denote by $\mathcal{T}_h$ a subdivision of the domain $\bar \Omega$ into shape regular 
tetrahedra $K$ of diameter $h_K$ and unit outward normal to $\partial K$ given by $\n_K$. 
We point out 
that the partition $\mathcal{T}_h$ is not necessarily a conforming mesh of $\bar \Omega$. 
We also  introduce a  shape regular conforming quasi-uniform triangulation $\cG_h:=\set{T}$ of the interface $\Gamma$
into triangles $T$ of diameter $h_T$. The set of edges of $\cG_h$ is denoted by $\cE_h$.
%For simplicity, we will assume that there exist constants $c_1, c_2>0$ such that 
%\[
% c_1 h \leq \max_{T \in \cG_h} h_T \leq c_2 h,
%\]
%so that we will only have to deal with one discretization parameter $h$. 
The parameter $h$ represents the mesh size, i.e.,
$h:= \max_{K\in \cT_h;\; T\in\cG_h} \{h_K, h_T\}$.

Henceforth, given any positive functions $A_h$ and $B_h$ of the mesh parameter $h$, 
the notation $A_h \lesssim B_h$ means that  $A_h \leq C B_h$ with $C>0$ independent of $h$ and 
$A_h \simeq B_h$ means that $A_h \lesssim B_h$ and $B_h \lesssim A_h$.

We say that a closed subset $F\in \overline{\Omega}$ is an interior face if $F$ has a positive 2-dimensional 
measure and if there are distinct elements $K$ and $K'$ such that $F = K\cap K'$. A closed 
subset $F\in \overline{\Omega}$ is a boundary face if
there exists $K\in \cT_h$ such that $F$ is a face of $K$ and $F = K\cap \Gamma$. 
We consider the set $\cF_h^0$ of interior faces and the set $\cF_h^\partial$ of boundary faces and introduce  
\[
 \cF_h = \cF_h^0\cup \cF_h^\partial.
\]
For any element $K\in \cT_h$, we introduce the set 
\[
\cF(K):= \set{F\in \cF_h;\quad F\subset \partial K} 
\]
of faces composing the boundary of $K$.
Similarly, for any $T\in \cG_h$, we introduce the set 
\[
 \cE(T):= \set{e\in \cE_h;\quad e\subset \partial T}.
\]
We also consider for any $T\in \cG_h$,
\[
 \cF(T):= \set{F\in \cF_h^\partial;\quad F\cap T\not=\emptyset}.
\]
% with triangles of diameter $h_K$. The vertices of $\cG_h$ does not necessarily coincide 
% of the boundary vertices of $\cT_h$. Actually, in contrast to the non-matching  interior mesh $\cT_h$, 
% we need to assume here that $\cG_h$ is a conforming simplicial triangulation.  
% We also assume that $h \simeq \max_{K\in \cG_h} h_K$.
%
In what follows we assume that $\cT_h\cup\cG_h$ is locally quasi-uniform,
i.e., there exists $\delta>1$ independent of $h$  such that 
$
 \delta^{-1} \leq \frac{h_K}{h_{K'}} \leq \delta  
$
for each pair $K$, $K'\in \cT_h$ sharing an interior face
and
$
 \delta^{-1} \leq \frac{h_K}{h_{T}} \leq \delta  
$
for each pair $K\in \cT_h$, $T\in \cG_h$ with $K\cap T\not=\emptyset$.
This assumption implies that the sets $\cF(K)$ and $\cF(T)$ have
uniformly bounded cardinalities and that there 
exists a constant $C>0$ independent of $h$ such that 
\begin{equation}\label{localUnif}
h_F \leq h_K \leq C \delta h_F \quad \forall F\in \cF(K)
\quad
\text{and}
\quad
h_F \leq h_T \leq C \delta h_F \quad \forall F\in  \cF(T),
\end{equation}
where $h_F$ stands for the diameter of the face $F$.

For any $s\geq 0$, we consider the broken Sobolev spaces 
\[
 H^s(\mathcal{T}_h) := \prod_{K\in \mathcal{T}_h} H^s(K), \qquad  
 \mathbf{H}^s(\mathcal{T}_h) := \prod_{K\in \mathcal{T}_h} H^s(K)^3, 
 %\qquad L^2(\cE_h):= \prod_{e\in \mathcal{E}_h} L^2(e)
 \]
 \[
 H^s(\mathcal{G}_h) := \prod_{T\in \mathcal{G}_h} H^s(T), \qquad  
 \mathbf{H}^s(\mathcal{G}_h) := 
\prod_{T\in \mathcal{G}_h} H^s(T)^3.
 %\mathbf{L}^2(\cE_h):= \prod_{e\in \mathcal{E}_h} L^2(e)^3.
\]
For each $v:=\set{v_K}\in H^s(\mathcal{T}_h)$, $\btau:= \set{\btau_K}\in \mathbf{H}^s(\mathcal{T}_h)$
and $\varphi:= \set{\varphi_T}\in \mathbf{H}^s(\cG_h)$, 
the components $v_K$, $\btau_K$ and $\varphi_T$ represent the restrictions $v|_K$, $\btau|_K$ and $\varphi|_T$. 
 When no confusion arises, the restrictions of these functions will be written 
without any subscript. 
The spaces $H^s(\mathcal{T}_h)$ and $\mathbf{H}^s(\mathcal{T}_h)$ are endowed with the Hilbertian norms
\[
 \norm{v}_{s,\cT_h}^2 := \sum_{K\in \cT_h} \norm{v_K}^2_{s,K} \qquad 
 \norm{\btau}_{s,\cT_h}^2 := \sum_{K\in \cT_h} \norm{\btau_K}^2_{s,K}.
\]
The corresponding seminorms are denoted by
\[
 |v|_{s,\cT_h}^2 := \sum_{K\in \cT_h} |v_K|^2_{s,K} \qquad 
 |\btau|_{s,\cT_h}^2 := \sum_{K\in \cT_h} |\btau_K|^2_{s,K}.
\]
Similarly, the norms and the seminorms on $H^s(\mathcal{G}_h)$  are given by
\[
 \norm{\varphi}_{s,\cG_h}^2 := \sum_{T\in \cG_h} \norm{\varphi_T}^2_{s,T} \qquad 
 |\varphi|_{s,\cG_h}^2 := \sum_{T\in \cG_h} |\varphi_T|^2_{s,T}.
\]
Identical definition for the norms and the seminorms are considered on the vectorial counterpart of $H^s(\mathcal{G}_h)$.
We use the convention $H^0(\mathcal{T}_h)= L^2(\mathcal{T}_h)$ for all the spaces defined previously.

%$H^s(\mathcal{T}_h) := \prod_{K\in \mathcal{T}_h} H^s(K)$,
%$\mathbf{H}^s(\mathcal{T}_h) := \prod_{K\in \mathcal{T}_h} H^s(K)^3$, 
% $H^s(\mathcal{G}_h) := \prod_{T\in \mathcal{G}_h} H^s(T)$ and $\mathbf{H}^s(\mathcal{G}_h) := 
% \prod_{T\in \mathcal{G}_h} H^s(T)$.
We will also need the spaces given on the skeletons of the triangulations $\cT_h$ and $\cG_h$ by 
\[
 L^2(\cF_h):= 
 \prod_{F\in \mathcal{F}_h} L^2(F), \qquad  L^2(\cF_h^0):= 
 \prod_{F\in \mathcal{F}_h^0} L^2(F), \qquad L^2(\cE_h):= 
 \prod_{e\in \mathcal{E}_h} L^2(e)
 \]
 \[
 \mathbf{L}^2(\cF_h):= \prod_{F\in \mathcal{F}_h} L^2(F)^3, \qquad 
 \mathbf{L}^2(\cF_h^0):= \prod_{F\in \mathcal{F}_h^0} L^2(F)^3, \qquad 
 \mathbf{L}^2(\cE_h):= \prod_{e\in \mathcal{E}_h} L^2(e)^3.
\]
Similarly, the components $\mu_F$ and $\bbeta_F$ of $\mu := \set{\mu_F}\in L^2(\cF_h)$ and 
$\bbeta:= \set{\bbeta_F} \in \mathbf{L}^2(\cF_h)$ coincide with the restrictions $\mu|_F$ and 
$\bbeta|_F$ and the components $\varphi_e$ and $\boldsymbol{\psi}_e$ of 
$\varphi := \set{\varphi_e}\in L^2(\cE_h)$ and $\boldsymbol{\psi}:=\set{\boldsymbol{\psi}_e}\in 
\mathbf{L}^2(\cE_h)$ 
are given by the restrictions $\varphi|_e$ and $\boldsymbol{\psi}|_e$
respectively.
We introduce the  inner products
\[
 \dual{\lambda, \mu}_{\cF_h} := \sum_{F\in \cF_h} \dual{\lambda_F, \mu_F}_F, \quad 
 \dual{\lambda, \mu}_{\cF^0_h} := \sum_{F\in \cF^0_h} \dual{\lambda_F, \mu_F}_F \quad 
 \text{and} \quad \dual{\psi, \varphi}_{\cE_h} := \sum_{e\in \cE_h} \dual{\psi_e, \varphi_e}_e.
 \]
 and the corresponding norms 
 \[
  \norm{\mu}^2_{0,\cF_h}:= \dual{\mu, \mu}_{\cF_h}, 
\qquad 
  \norm{\mu}^2_{0,\cF^0_h}:= \dual{\mu, \mu}_{\cF^0_h}\quad  \text{and} \quad 
  \norm{\varphi}^2_{0,\cE_h}:= \dual{\varphi, \varphi}_{\cE_h}
 \]
 on $L^2(\cF_h)$, $L^2(\cF^0_h)$ and $L^2(\cE_h)$ respectively.
 
Given $v\in H^1(\cT_h)$, we define averages $\mean{v}\in L^2(\cF_h^0)$ and jumps $\jump{v}\in \mathbf{L}^2(\cF_h^0)$ 
by
\[
 \mean{v}_F := 1/2(v_K + v_{K'}) \quad \text{and} \quad \jump{v}_F := v_K \n_K + v_{K'}\n_{K'} 
 \quad \forall F \in \cF(K)\cap \cF(K').
\]
For vector valued functions $\btau\in \mathbf{H}^1(\cT_h)$, we define $\mean{\btau}\in \mathbf{L}^2(\cF_h^0)$ and 
$\jump{\btau}\in L^2(\cF_h^0)$ by
\[
 \mean{\btau}_F := 1/2(\btau_K + \btau_{K'}) \quad \text{and} \quad \jump{\btau}_F := 
 \btau_K\cdot \n_K + \btau_{K'}\cdot\n_{K'} 
 \quad \forall F \in \cF(K)\cap \cF(K').
\]
Similarly, given $\varphi\in H^1(\cG_h)$, we define averages $\mean{\varphi}\in L^2(\cE_h^0)$ and jumps 
$\jump{\varphi}\in \mathbf{L}^2(\cE_h)$ by 
\[
 \mean{\varphi}_e := 1/2(\varphi_T + \varphi_{T'}) \quad \text{and} \quad \jump{\varphi}_e := 
 \varphi_T \tg_e + \varphi_{T'} \tg_{e'} 
 \quad \forall e \in \cE(T)\cap \cE(T').
\]
Here, $\tg_e$ is the tangent unit vector along the edge $e$ given by 
$\tg_e = (\n \times \n_{\partial F})|_e$,  where $\n_{\partial F}$ is 
the   outward unit  normal vector to the boundary of the face $F$ in the 
hyperplane defined by $\n|_F$. 

Hereafter, given an integer $k\geq 0$ and a domain 
$D\subset \mathbb{R}^3$, $\cP_k(D)$ denotes the space of polynomials of degree at most $k$ on $D$.
We consider the linear spaces
\[
 \cP_0(\cF_h):= \prod_{F\in \mathcal{F}_h} \cP_0(F) \quad \text{and} \quad  \cP_0( \cF_h^0):= 
 \prod_{F\in \mathcal{F}_h^0} \cP_0(F),
\]
and for any $m\geq 1$, we introduce the finite element spaces
\[
 V_h := \prod_{K\in \mathcal{T}_h} \cP_m(K) \quad \text{and} \quad 
 \bSigma_h := \prod_{K\in \cT_h} \mathbf{RT}_m(K),
\]
where 
\[
 \mathbf{RT}_m(K):=\set{\cP_{m-1}(K)^3 + \x\cP_{m-1}(K)}
\]
is the finite element of Raviart-Thomas of order $m-1$.

We consider the following formulation in the bounded domain $\Omega$: find $(\bsigma_h, u_h) \in \bSigma_h\times V_h$ 
such that for each $K\in \cT_h$ there holds
\begin{equation}\label{formT1}
 \begin{array}{rcll}
(\bsigma_h, \btau)_K -  (\nabla u_h, \btau)_K + \dual{u_h -\bar u, \btau \cdot \n_K}_{\partial K} &=& 0 & 
  \forall \btau \in \bSigma_h\\[2ex]
\disp  (\bsigma_h, \nabla v)_K - \dual{ \bar \bsigma \cdot \n_K, v}_{\partial K} &=&  (f, v)_K & \forall v \in V_h.
 \end{array}
\end{equation}
Before defining  the numerical traces $\bar u$ and $\bar \bsigma$ let us consider the finite element approximation of 
the boundary integral equations \eqref{inteq1} and \eqref{inteq2}. 
 
We consider the operator $\boldsymbol T$ defined for any $\varphi \in H^1(\cG_h)$ by 
\[
 (\boldsymbol T\varphi)|_e := (V \curl_h \varphi)|_e  \quad (e\in \cE_h),
\]
where $\curl_h$ stands for the element-wise $\curl$ operator:
\[
 (\curl_h \varphi)|_F := \curl_F(\varphi|_F), \qquad \forall F\in \cG_h.
\]
We consider two sequences of boundary element spaces 
\begin{align*}
 \Lambda_h &:= \set{\btau\cdot \n;\quad  \btau \in \bSigma_h } \subset H^{-1/2}(\Gamma),\\
 \Psi_h    &:= \prod_{T\in \mathcal{G}_h} \cP_m(T) \cap L^{2}_0(\Gamma)
 \subset H_0^{1/2}(\cG_h):=H^{1/2}(\cG_h)\cap L^2_0(\Gamma) 
\end{align*}
with $L^{2}_0(\Gamma):=\set{\varphi\in L^{2}(\Gamma); \quad \dual{1,\varphi}_\Gamma = 0}$.
We then replace \eqref{inteq1}, \eqref{inteq2} by the Galerkin equations:
find $\psi_h\in \Psi_h$, $\lambda^\star_h\in \Lambda_h$ such that 
\begin{equation}\label{formB}
 \begin{array}{rcll}
 \dual{ \psi_h, \btau\cdot \n}_{\Gamma} &=& \dual{( \frac{\text{id}}{2} + K) \psi_h, \btau\cdot \n}_{\Gamma} - 
 \dual{V (\lambda^\star_h -  g_1), \btau\cdot \n}_{\Gamma} 
 &\forall \btau \in \bSigma_h\\[2ex]
\dual{(\lambda^\star_h - g_1), \varphi}_{\Gamma} &=& -d(\psi_h,\varphi) + 
\dual{(\frac{\text{id}}{2} - K') (\lambda^\star_h-g_1), \varphi}_{\Gamma} 
&\forall \varphi \in \Psi_h.
 \end{array}
\end{equation}
Here we used the transmission condition for $\lambda$, \eqref{trans}, and $\lambda^\star_h$ will be an
approximation to $\lambda+g_1=\bsigma\cdot\n$. Furthermore,
\begin{equation}\label{d}
 d(\psi, \varphi) := \dual{V\curl_h \psi, \curl_h\varphi }_\Gamma + \dual{\boldsymbol T\psi, \jump{\varphi}}_{\cE_h} 
 - \dual{\jump{\psi}, \boldsymbol T \varphi}_{\cE_h} + \dual{\nu\jump{\psi}, \jump{\varphi}}_{\cE_h}
\end{equation}
and $\nu\in \prod_{e\in \cE_h} \cP_0(e)$ is a piecewise constant function 
such that 
\begin{equation}\label{nu}
\nu\simeq 1.  
\end{equation}
Let $\alpha\in \cP_0(\cF_h)$,  and $\bbeta \in \cP_0( \cF_h^0)^3$ be given 
piecewise constant functions  satisfying
\begin{equation}\label{coef}
\max_{F\in \cF_h^0} |\bbeta_F| \lesssim 1 \quad \text{and} 
\quad   h_\cF\, \alpha \simeq 1,
\end{equation}
where $h_\cF\in \cP_0(\cF_h)$ is defined by $h_\cF|_F := h_F$ $, \forall F \in \mathcal{F}_h$.

We substitute $\bar u$ and $\bar \bsigma$ given by
\[
 \bar u_F = \begin{cases}
           \mean{u_h}_F + \bbeta_F\cdot \jump{u_h}_F &\text{if $F\in \cF_h^0$}\\[2ex]
           \psi_h+g_0 &\text{if $F\in \cF_h^{\partial}$}
          \end{cases}
\]
and
\[
 \bar \bsigma_F = \begin{cases}
                   \mean{\bsigma_h}_F - \jump{\bsigma_h}_F \bbeta_F - \alpha_F \jump{u_h}_F &\text{if $F\in \cF_h^0$}\\[2ex]
                   \bsigma_h|_F - \alpha_F (u_h|_F-\psi_h-g_0) \n|_F &\text{if $F\in \cF_h^{\partial}$}
                  \end{cases}
\]
in \eqref{formT1} and add the equations over $K\in \cT_h$ to obtain the 
following LDG formulation of the problem in $\Omega$: find $(\bsigma_h,u_h)\in \bSigma_h\times V_h$ such that 
\begin{equation}\label{ldgOmega}
 \begin{array}{rcll}
(\bsigma_h, \btau)_{\Omega} - \left\{ (\nabla_h u_h, \btau)_{\Omega} - S(u_h, \btau) \right\}  
   - \dual{\psi_h, \btau\cdot \n}_\Gamma &=& \dual{g_0, \btau\cdot \n}_\Gamma \\[2ex]
\left\{(\nabla_h v, \bsigma_h)_{\Omega}  - S(v, \bsigma_h)\right\} + \boldsymbol{\alpha}_0(u_h, v)  
+\dual{\alpha (u_h - \psi_h), v}_\Gamma &=& (f,v)_{\Omega} + \dual{\alpha g_0, v}_\Gamma,
 \end{array}
\end{equation}
for all $\btau \in \bSigma_h$ and $v\in V_h$, where $\nabla_h$ stands for the element-wise gradient and 
\[
 S(u, \btau) := \dual{\jump{u}, \mean{\btau} - \jump{\btau}\bbeta}_{\cF_h^0} + \dual{u, \btau\cdot \n}_\Gamma, 
 \qquad \forall u\in H^1(\cT_h),\, \forall \btau \in \mathbf{H}^1(\cT_h),
\]
\[
 \boldsymbol{\alpha}_0(u,v) = \dual{ \alpha \jump{u}, \jump{v}}_{\cF_h^0}   
 \qquad \forall u,v\in H^1(\cT_h),
\] 
In order to simplify the notations, let us denote by $\hat{u}_h=(u_h,\psi_h)$ and $\hat{v}:=(v,\varphi)$ couples 
of elements from $V_h \times \Psi_h$. We also consider 
\[
 \jump{\hat{u}_h}: = \begin{cases}
                \jump{u_h}_F & \text{if $F\in \cF_h^0$}\\
                (u_h - \psi_h)\n|_F & \text{if $F\in \cF_h^{\partial}$}
               \end{cases}
               \quad 
               \text{and}
               \quad
               \jump{\hat{v}}: = \begin{cases}
                \jump{v}_F & \text{if $F\in \cF_h^0$}\\
                (v - \varphi)\n|_F & \text{if $F\in \cF_h^{\partial}$}
               \end{cases}.
\]
We now couple \eqref{formB} and \eqref{ldgOmega} by
identifying $\lambda^\star_h=\bsigma_h\cdot\n$ and by approximating the transmission condition for the traces 
in \eqref{trans} by
\begin{equation*}
\dual{\alpha (u_h - \psi_h), \varphi}_\Gamma = \dual{\alpha g_0, \varphi}_\Gamma \quad \forall \varphi \in \Psi_h.
\end{equation*}
A combination of \eqref{formB} and \eqref{ldgOmega} then yields our LDG-FEM/DG-BEM coupling:
find $(\bsigma_h, \hat{u}_h)\in \bSigma_h\times (V_h \times \Psi_h)$ such that
\begin{equation}\label{ldg-FemBem}
 \begin{array}{rcll}
 a(\bsigma_h,\btau) + b(\btau, \hat{u}_h) &=& \dual{g_0, \btau\cdot \n}_{\Gamma} + 
 \dual{Vg_1, \btau \cdot \n}_{\Gamma} &\forall \btau\in \bSigma_h\\[2ex]
 -b(\bsigma_h, \hat{v})+ c(\hat{u}_h, \hat{v}) &=& (f,v)_{\Omega} + \dual{\alpha g_0, v-\varphi}_{\Gamma}
 +\dual{( \frac{\text{id}}{2} + K')g_1, \varphi}_{\Gamma}
 & \forall \hat{v}\in V_h\times \Psi_h.
 \end{array}
\end{equation}
Here,
\[
 a(\bsigma_h,\btau):= (\bsigma_h, \btau)_{\Omega} + \dual{\btau\cdot \n, V (\bsigma_h\cdot \n)}_{\Gamma},
\quad 
 c(\hat{u}_h, \hat{v}) :=  \dual{\alpha \jump{\hat{u}_h},\jump{\hat{v}}}_{\cF_h} + d(\psi_h, \varphi)
\]
and
\[
 b(\btau, \hat{v}) := -(\nabla_h v, \btau)_{\Omega} + \dual{\btau\cdot \n, ( \frac{\text{id}}{2} - K) \varphi}_{\Gamma} 
 +\dual{\jump{v}, \mean{\btau} - \jump{\btau}\bbeta}_{\cF_h^0} + \dual{ \jump{\hat{v}}, \btau}_{\Gamma}.
\]
Problem \eqref{ldg-FemBem} can be rewritten in the more compact form as follows: Find $\bsigma_h\in \bSigma_h$ and 
$\hat{u}_h=(u_h, \psi_h)\in V_h\times \Psi_h$ such that 
\begin{equation}\label{compactForm}
 A(\bsigma_h, \hat{u}_h; \btau, \hat{v}) = F(\btau, \hat{v}),
\end{equation}
by setting 
\begin{equation}\label{A}
 A(\bsigma_h, \hat{u}_h; \btau, \hat{v}):= a(\bsigma_h,\btau) + b(\btau, \hat{u}_h) - b(\bsigma_h, \hat{v}) + c(\hat{u}_h, \hat{v})
\end{equation}
and 
\[
 F(\btau, \hat{v}):= (f,v)_{\Omega}+ \dual{Vg_1 + g_0, \btau\cdot \n}_{\Gamma} + \dual{\alpha g_0, v-\varphi}_{\Gamma}
 +\dual{( \frac{\text{id}}{2} + K')g_1, \varphi}_{\Gamma}.
\]

\begin{prop}\label{wellposed}
The LDG-FEM/DG-BEM method defined by \eqref{ldg-FemBem} provides a unique approximate solution $(\bsigma_h, (u_h,\psi_h))
\in \bSigma_h\times (V_h\times \Psi_h)$.
\end{prop}
\begin{proof}
 It suffices to prove that if $f=0$, $g_0=0$ and $g_1=0$, 
 then \eqref{ldg-FemBem} admits only the trivial solution. Taking $\btau = \bsigma_h$ 
 and $\hat{v}=\hat{u}_h$ in \eqref{compactForm} yields
 \[
  a(\bsigma_h, \bsigma_h) + c(\hat{u}_h,\hat{u}_h) = 0,
 \]
which proves that $\bsigma_h=\mathbf{0}$, $\jump{\hat{u}_h}=0$, $\curl_h \psi_h =0$  
and $\jump{\psi_h}_e=0$ for all $e\in \cE_h$. Consequently, $\psi_h$ is constant 
on $\Gamma$ and, as it has zero mean value, it must vanish identically. 
Now, $\psi_h=0$ and $\jump{\hat{u}_h}=0$ implies that $u_h=0$ on $\Gamma$.
 On the other hand, it follows from
\[
 b(\btau, \hat{u}_h) = -( \nabla_h u_h,  \btau)_{\Omega} = 0\quad \forall \btau \in \bSigma_h
\]
and the fact that $\nabla_h(V_h)\subset \bSigma_h$ that $\nabla_h u_h=\mathbf{0}$. We can now conclude that
$u_h=0$ since it is constant in each $T$, it has no jumps across the interior faces of $\cT_h$ 
($\jump{u_h}_F = 0$ for all $F\in \cF_h^0$) and 
it vanishes on $\Gamma$. 
\end{proof}

We end this section by proving that our LDG-FEM/DG-BEM scheme is consistent. 
\begin{prop}\label{consistency0}
 Let $u$ be the solution of \eqref{ModelProblem} in $\Omega$, $\bsigma:= \nabla u$ and $\psi:=u|_\Gamma - g_0$. Under the regularity 
 assumptions $u\in H^2(\Omega)$ and $W \psi\in L^2(\Gamma)$ we have that  
 \[
  A(\bsigma, (u, \psi); \btau, (v,\varphi)) = F(\btau, (v, \varphi))\quad \forall \btau \in \bSigma_h, \quad 
  \forall (v,\varphi)\in V_h\times \Psi_h.
 \]
\end{prop}
\begin{proof}
 Taking into account that $\lambda:= \bsigma\cdot\n - g_1 = \nder{u} - g_1$, it is straightforward to show that 
 \begin{multline*}
  A(\bsigma, (u, \psi); \btau, (v,\varphi)) = \dual{\btau\cdot \n, g_0 + Vg_1}_\Gamma + 
  \dual{\alpha g_0, v - \varphi}_\Gamma + \dual{(\frac{\text{id}}{2} + K') g_1, \varphi}_\Gamma \\[1ex]
  +\dual{\btau\cdot \n, V\lambda + (\frac{\text{id}}{2} - K)\psi}_\Gamma 
  + \dual{(\frac{\text{id}}{2} + K') \lambda, \varphi}_\Gamma+ \dual{V \curl_\Gamma \psi, \curl_h \varphi}_{\Gamma} +
   \dual{\boldsymbol T \psi, \jump{\varphi}_e}_{\cE_h}\\[1ex]
  + \sum_{K\in \cT_h} (\nabla u, \nabla v)_K  
  -\dual{\nabla u, \jump{v} }_{\cF^0_h}  - \dual{\nder{u}, v}_\Gamma.
 \end{multline*}
Taking into account the integration by parts formula
 \begin{multline*}
  \dual{V \curl_\Gamma \psi, \curl_h \varphi}_{\Gamma} + \dual{\boldsymbol T \psi, \jump{\varphi}_e}_{\cE_h}= \\[1ex]
  \sum_{T\in \cG_h} \dual{V \curl_\Gamma \psi, \curl_T \varphi}_{T} 
  +  \dual{ \tg_{\partial T}\cdot V \curl_\Gamma \psi, \varphi  }_{\partial T}=
  \dual{W \psi, \varphi},
 \end{multline*}
we deduce the result from \eqref{inteq1}, \eqref{inteq2} and from the fact that
\[
 \sum_{K\in \cT_h} (\nabla u, \nabla v)_K = \sum_{K\in \cT_h} (f, v)_K 
 +\dual{\nabla u, \jump{v} }_{\cF^0_h}  + \dual{\nder{u}, v}_\Gamma.
\]
\end{proof}

\section{Convergence analysis}\label{s4}
In this section, we develop the error analysis of the LDG-FEM/DG-BEM scheme \eqref{ldg-FemBem}. We first 
introduce a series of technical results that are used in the proof the C\'ea's error estimate provided by 
Theorem \ref{cea}). Then, we use well-known interpolation error estimates to 
obtain the main convergence result stated in Theorem \ref{main}. 

\subsection{Technical results}
The following discrete trace inequality is standard, \cite{DiPietroErn}.
\begin{lemma}
 For all $K\in \cT_h$, all integer $k\geq 0$, and all $v\in \cP_k(K)$, 
 \begin{equation}\label{discreteTrace}
  h_K \norm{v}^2_{0,\partial K} \lesssim \norm{v}^2_{0,K}.
 \end{equation}
\end{lemma}

\begin{prop}\label{card}
 For all $v\in H^1(\cT_h)$, 
 \[
  \norm{\jump{v}}^2_{0,\cF^0_h}\lesssim \sum_{K\in \cT_h} \norm{v}^2_{0,\partial K}.
 \]
\end{prop}
\begin{proof}
 The proof relies on the local quasi-uniformity of $\cT_h$. Indeed,
 \[
  \norm{\jump{v}}^2_{0,\cF^0_h} = \sum_{F\in \cF^0_h} \norm{\jump{v}}^2_{0,F} \leq 2  
  \sum_{K\in \cT_h} \sum_{F\in \cF(K)} \norm{v}^2_{0,F},
\]
and the result follows from the fact that the cardinality of the set $\cF(K)$ is uniformly bounded.
\end{proof}
The  $H^{1/2}(\cG_h)$-ellipticity of the  bilinear form $\dual{V\curl_h \psi, \curl_h\varphi }_\Gamma$ in 
$\Psi_h$ is essential for the stability of our method. The main difficulty that we had to deal with 
in our analysis is that this bilinear form is not uniformly bounded on $\Psi_h$ with respect 
to this broken-norm.  
\begin{lemma}\label{lemaVcurl}
 There holds
\[
\dual{V\curl_h \varphi, \curl_h\varphi }_\Gamma \gtrsim |\varphi|^2_{1/2,\cG_h},\qquad \forall \varphi\in\Psi_h.
\]
\end{lemma}
\begin{proof}
The result is a consequence of \eqref{eq-coercive-1} and the fact that  
(cf. \cite{HeuerSayas})
\begin{equation}\label{eqnorm}
 |\varphi|^2_{1/2, \cG_h}\lesssim  \sum_{T\in \cG_h} \norm{\curl_T \varphi}^2_{-1/2, T}\leq \norm{\curl_h \varphi}^2_{-1/2, \Gamma}
 \quad \forall \varphi \in \Psi_h.
\end{equation}
\end{proof}

The following estimate is a Poincar\'e-Friedrichs inequality for piecewise polynomial functions.  
\begin{lemma}
There holds
 \begin{equation}\label{discP}
 \norm{ \varphi}^2_{0,\Gamma} \lesssim \left( |\log{h}|
 |\varphi|^2_{1/2, \cG_h} + \norm{\jump{\varphi}}^2_{0,\cE_h} 
 \right) \qquad \forall \varphi\in\Psi_h.
\end{equation}
\end{lemma}
\begin{proof}
 We know from \cite[Theorem 8]{HeuerSayas0} that 
 \[
  \norm{\varphi}^2_{0,\Gamma} \lesssim \varepsilon^{-1} \abs{\varphi}^2_{1/2 + \varepsilon, \cG_h} + 
  \sum_{e \in \cE_h} h_e^{-1-2\epsilon} \left|\int_e \jump{\varphi}_e \right|^2 + \left| \int_\Gamma \varphi\right|^2
 \]
for all $\varphi \in H^{1/2 + \epsilon}(\cG_h)$ and for all $\varepsilon \in (0, 1/2)$. The inverse inequality
\[
 \abs{\varphi_T}_{1/2 + \varepsilon, T} \lesssim h_T^{-\varepsilon}\, \abs{\varphi_T}_{ 1/2, T} \quad 
 \forall \varphi_T \in \cP_m(T),
\]
and the fact that $\cG_h$ is quasi-uniform yields 
\[
  \norm{\varphi}^2_{0,\Gamma} \lesssim \varepsilon^{-1} h^{-2\varepsilon}\abs{\varphi}^2_{1/2 , \cG_h} + 
  h^{-2\varepsilon}\sum_{e \in \cE_h} h_e^{-1} \left|\int_e \jump{\varphi}_e \right|^2 
 \]
for all $\varphi \in \Psi_h$ and for all $\varepsilon \in (0, 1/2)$. The result follows now by choosing 
$\varepsilon = \frac{1}{2} (\log 1/h)^{-1}$ and applying the Cauchy-Schwarz inequality. 
\end{proof}

Finally, the following bound for $\boldsymbol T$ can be found in \cite[Equation (4.27)]{NorbertSalim}.
\begin{lemma}\label{cotaT}
We have that,  
 \[
  \norm{\boldsymbol T \varphi}_{0,\cE_h} \lesssim\,  \left( h^{-1}\norm{\varphi}^2_{0,\Gamma} +  
  h |\varphi|^2_{1,\Gamma} \right)^{1/2},\qquad \text{for all $\varphi \in H^1(\Gamma)$}. 
 \]
\end{lemma}

\subsection{Stability of the LDG-FEM/DG-BEM method}

For all $K\in \cT_h$, we introduce the $L^2(K)$-orthogonal projector $\Pi_K$ onto $\cP_m(K)$. Moreover, we consider 
on each $T\in \cG_h$ the usual triangular Lagrange finite element of order $m$ $(m\geq 1)$ and 
denote by  $\tilde\pi_T:\, \mathcal{C}^0(T)\to \cP_m(T)$ the corresponding Lagrange interpolation operator.
We will also use the Raviart-Thomas interpolation operator $\bPi_K$ in $\mathbf{RT}_m(K)$, 
see \cite{rt91}. The global 
operators $\Pi:\, L^2(\cT_h) \to V_h$, $\bPi:\, \mathbf{H}^1(\cT_h) \to \bSigma_h$  and $\tilde\pi:\, \mathcal{C}^0(\Gamma) 
\to \Psi_h\cap \mathcal{C}^0(\Gamma)$ 
 are given by 
\[
(\Pi v )|_K := \Pi_K (v_K), \quad  (\bPi\btau )|_K := \bPi_K (\btau_K)\,\, \forall K \in \cT_h
\quad \text{and}\quad (\tilde\pi \varphi)|_T := \tilde\pi_T (\varphi_T)\,\,  \forall T\in \cG_h
\]
respectively.

For all $\btau \in \mathbf{H}^1(\cT_h)$ and $\hat{v}:= (v,\varphi)\in H^1(\cT_h)\times H^{1}(\cG_h)$, we  introduce the semi-norms
\[
 \norm{(\btau, \hat{v})} := \left(\norm{\btau}^2_{0,\Omega} + \norm{\btau\cdot\n}^2_{-1/2,\Gamma} + 
 |\varphi|^2_{1/2,\cG_h} + \norm{\alpha^{1/2} \jump{\hat{v}}}^2_{0,\cF_h} + 
 \norm{\nu^{1/2}\jump{\varphi}}^2_{0, \cE_h}\right)^{1/2} ,
\]
\begin{multline*}
 \norm{(\btau, \hat{v})}_{\#} := \Big(\norm{\btau}^2_{0,\Omega} + \norm{\btau\cdot\n}^2_{-1/2,\Gamma}  
 + \norm{\alpha^{1/2} \jump{\hat{v}}}^2_{0,\cF_h} + \\[1ex]
  \norm{\curl _h\varphi}^2_{-1/2,\Gamma}+ \norm{\nu^{1/2}\jump{\varphi}}^2_{0, \cE_h}\Big)^{1/2}
\end{multline*}
and for all $\btau \in \mathbf{H}^1(\cT_h)$ and $\hat{v}:= (v,\varphi)\in H^1(\cT_h)\times H^{1}(\Gamma)$, 
we introduce 
\begin{multline*}
 \norm{(\btau, \hat{v})}_*:= \Big(\norm{(\btau, \hat{v})}^2+
 \sum_{K\in \cT_h}\norm{\alpha^{-1/2} \btau\cdot \n_K}^2_{0,\partial K}+ \norm{ \btau\cdot \n}^2_{0,\Gamma}+
\sum_{K\in \cT_h} \norm{\alpha^{1/2} v}^2_{0,\partial K} +\\[1ex]
 \norm{\varphi}^2_{1/2,\Gamma}+
h^{-1} \norm{\varphi}^2_{0,\Gamma}+h |\varphi|^2_{1,\Gamma}\Big)^{1/2}.
\end{multline*}
It is clear that 
\begin{equation}\label{rn1}
\norm{(\btau, \hat{v})} \leq \norm{(\btau, \hat{v})}_*\quad 
\forall (\btau, \hat{v}) \in \mathbf{H}^1(\cT_h)\times ( H^1(\cT_h)\times H^{1}(\Gamma) ).
\end{equation}
Moreover, taking into account \eqref{eqnorm}, 
we deduce that  
\begin{equation}\label{rn2}
 \norm{(\btau, \hat{v})} \lesssim \norm{(\btau, \hat{v})}_{\#} \quad 
 \forall (\btau, \hat{v}) \in \mathbf{H}^1(\cT_h)\times ( H^1(\cT_h)\times H^{1}(\cG_h)).
\end{equation}
%For any $\varphi \in H^1(\cG_h)$ we consider the seminorm
%\[
% |\varphi|^*_{1/2,\cG_h} = \left( \sum_{T\in \cG_h} |\varphi|^2_{1/2,T} + \nu \norm{\jump{\varphi}}^2_{0,\cE_h}
% \right)^{1/2}.
%\]
In the following we abbreviate
\[
 \pi_{\bsigma} := \bsigma - \bPi\bsigma, \quad
 \pi_u := u - \Pi u,\quad   \tilde\pi_\psi := \psi - \tilde\pi\psi 
\quad\text{and}\quad 
\pi_{\hat{u}} := (u - \Pi u, \psi - \tilde\pi \psi).
\]

\begin{lemma}\label{boundA}
Let us assume that $\bsigma \in H^{1/2+\varepsilon}(\Omega)^3$ with $\varepsilon> 0$ and $\psi \in H^{1}(\Gamma)$.
 Then, there exists a constant $\bar C>0$ independent of $h$ such that 
 \[
  \abs{A(\pi_{\bsigma}, \pi_{\hat{u}}; \btau, \hat{v})}\leq \bar C |\log{h}|^{1/2}\,  
  \norm{(\pi_{\bsigma}, \pi_{\hat{u}})}_* \norm{(\btau, \hat{v})}_{\#}
  \quad \forall  (\btau, \hat{v})\in \bSigma_h \times (V_h\times \Psi_h).
 \]
\end{lemma}
\begin{proof}
First of all, the definition of $A(\cdot, \cdot)$ and the triangle inequality yield  
\[
\abs{A(\pi_{\bsigma}, \pi_{\hat{u}}; \btau, \hat{v})}\leq 
\abs{a(\pi_{\bsigma},\btau)} + \abs{c(\pi_{\hat{u}}, \hat{v})} + 
\abs{ b(\pi_{\bsigma}, \hat{v}) } + \abs{ b(\btau, \pi_{\hat{u}}) } =: T_1 + T_2 + T_3 + T_4
\]
for all $(\btau, \hat{v})\in \bSigma_h \times (V_h\times \Psi_h)$.
Using Cauchy-Schwarz's inequality it is straightforward to see that
\begin{multline}\label{T1}
 T_1 \lesssim \, \left( \norm{\pi_{\bsigma}}^2_{0,\Omega} + \norm{\pi_{\bsigma}\cdot \n}^2_{-1/2,\Gamma} \right)^{1/2}
 \left( \norm{\btau}^2_{0,\Omega} + \norm{\btau\cdot \n}^2_{-1/2,\Gamma} \right)^{1/2}\\[1ex]
 \lesssim \norm{(\pi_{\bsigma}, \pi_{\hat{u}})} \norm{(\btau, \hat{v})}.
\end{multline}
Applying Lemma \ref{cotaT} and the Cauchy-Schwarz inequality we deduce that  
\begin{multline*}
 T_2 \leq 
 |\dual{\alpha \jump{\pi_{\hat{u}}},\jump{\hat{v}}}_{\cF_h}| + |d(\tilde\pi_\psi, \varphi)| \leq 
 |\dual{\alpha \jump{ \pi_{\hat{u}}},\jump{\hat{v}}}_{\cF_h}| +\\[1ex]
   |\dual{V\curl_h \tilde\pi_\psi, \curl_h\varphi }_\Gamma| + |\dual{\boldsymbol T\tilde \pi_\psi, \jump{\varphi}}_{\cE_h}| 
 + |\dual{\jump{\tilde\pi_\psi}, \boldsymbol T \varphi}_{\cE_h}| +  |\dual{\nu\jump{\tilde \pi_\psi}, \jump{\varphi}}_{\cE_h}|
 \\[1ex]
 = |\dual{\alpha \jump{\pi_{\hat u}},\jump{\hat{v}}}_{\cF_h}| +|\dual{V\curl_\Gamma \tilde \pi_\psi, \curl_h\varphi }_\Gamma| 
 + |\dual{\boldsymbol T\tilde\pi_\psi, \jump{\varphi}}_{\cE_h}| 
 \\[1ex]
 \lesssim 
 \norm{\alpha^{1/2} \jump{ \pi_{\hat{u}}}}_{0,\cF_h}\norm{\alpha^{1/2} \jump{\hat{v}}}_{0,\cF_h}+
  \norm{\curl_\Gamma \tilde\pi_\psi}_{-1/2,\Gamma} \norm{\curl_h\varphi}_{-1/2,\Gamma} + \\[1ex] 
 \norm{\nu^{1/2}\boldsymbol T\tilde\pi_\psi}_{0,\cE_h} \norm{\nu^{1/2}\jump{\varphi}}_{0,\cE_h}.
\end{multline*}
Taking advantage of the fact that $\curl_\Gamma:\, H^{1/2}(\Gamma)\to H^{-1/2}(\Gamma)^3$ is bounded we conclude that 
\begin{equation}\label{T2}
  T_2   \lesssim \norm{(\pi_{\bsigma}, \pi_{\hat{u}})}_* \norm{(\btau, \hat{v})}_{\#}.
\end{equation}
By definition of the Raviart-Thomas 
interpolation operator, $(\pi_{\bsigma}, \nabla v)_K = 0 $ for all $v\in V_h$, which implies that  
\[
 T_3 = \Bigl| \dual{\pi_{\bsigma}\cdot \n, ( \frac{\text{id}}{2} - K) \varphi}_\Gamma
 +\dual{\jump{v}, \mean{\pi_{\bsigma}} - \jump{\pi_{\bsigma}}\bbeta}_{\cF_h^0}
 + \dual{\jump{\hat{v}}, \pi_{\bsigma}}_\Gamma\Bigr|.
\]
We apply the Cauchy-Schwarz inequality and hypothesis \eqref{coef} on $\alpha$
to deduce that   
\begin{multline*}
 T_3 \lesssim \left( 
 \norm{\alpha^{-1/2}(\mean{\pi_{\bsigma}} - \jump{\pi_{\bsigma}}\bbeta)\cdot \n}^2_{0,\cF^0_h}
 + \max\{1, h_\cF\}\norm{ \pi_{\bsigma}\cdot \n}^2_{0,\Gamma} 
 \right)^{1/2}\\[1ex]
 \left( \norm{(\frac{\text{id}}{2} - K)\varphi}^2_{0,\Gamma} + \norm{\alpha^{1/2}\jump{\hat{v}}}^2_{0,\cF_h} \right)^{1/2}.
\end{multline*}
Hypothesis \eqref{coef} on $\bbeta$ and Proposition \ref{card} yield 
\[
 \norm{\alpha^{-1/2}(\mean{\pi_{\bsigma}} - 
 \jump{\pi_{\bsigma}}\bbeta)\cdot \n}^2_{0,\cF^0_h}
 \lesssim \sum_{K\in \cT_h} \norm{\alpha^{-1/2} \pi_{\bsigma}\cdot \n_K}^2_{0,\partial K}.
\] 
Moreover, the boundedness of 
$\frac{\text{id}}{2} - K:\, L^{2}(\Gamma)\to L^{2}(\Gamma)$ 
and the fractional order discrete Poincar\'e inequality \eqref{discP}
imply the estimate 
\begin{align*}
 \norm{  (\frac{\text{id}}{2} - K)\varphi}^2_{0,\Gamma}
 &\lesssim
 |\log{h}| \left( |\varphi|^2_{1/2, \cG_h} + \norm{\jump{\varphi}}^2_{0,\cE_h} \right)
 \\
 &\lesssim
 |\log{h}|
 \left( 
 \norm{\curl_h \varphi}^2_{-1/2, \Gamma} + \norm{\nu^{1/2}\jump{\varphi}}^2_{0,\cE_h} \right).
\end{align*}
This yields the following estimate,
\begin{equation}\label{T3}
 T_3 \lesssim| \log{h}|^{1/2}
 \left( \norm{(\pi_{\bsigma}, \pi_{\hat{u}})}^2 +   \norm{ \pi_{\bsigma}\cdot \n}^2_{0,\Gamma} +
 \sum_{K\in \cT_h}\norm{\alpha^{-1/2} \pi_{\bsigma}\cdot \n_K}^2_{0,\partial K}   
 \right)^{1/2}
\norm{(\btau, \hat{v})}_{\#}.
\end{equation}

To bound the last term $T_4$, we begin by using
integration by parts and the characterization of the $L^2(K)$-orthogonal projection onto $\cP_m(K)$,  
\[( \pi_u, \textrm{div} \btau)_K = 0\quad  \forall 
\btau \in \bSigma_h,
\]
to deduce that 
\[
 T_4 =  \Big | \dual{\jump{\pi_u}, \bbeta \jump{\btau} + \mean{\btau}}_{\cF_h^0} 
 + \dual{\btau\cdot \n, ( \frac{\text{id}}{2} + K) \tilde\pi_\psi}_\Gamma \Big |.
\]
Now, from the Cauchy-Schwarz inequality, the boundedness of $K:\, H^{1/2}(\Gamma)\to H^{1/2}(\Gamma)$,
the boundedness of $\bbeta$ and Proposition \ref{card}, it follows that   
\begin{multline*}
 T_4 \lesssim\, \left(\norm{\btau\cdot \n }^2_{-1/2,\Gamma} + 
 \sum_{K\in \cT_h}\norm{\alpha^{-1/2} \btau\cdot \n_K}^2_{0,\partial K}\right)^{1/2} \\[1ex]
 \left( \norm{\tilde\pi_\psi}_{1/2,\Gamma}^2 +   
\sum_{K\in \cT_h} \norm{\alpha^{1/2} \pi_u}^2_{0,\partial K}  \right)^{1/2}.
\end{multline*}
Finally, by virtue of \eqref{localUnif}, \eqref{coef} and \eqref{discreteTrace}, 
\[
 \norm{\alpha^{-1/2} \btau\cdot \n_K}^2_{0,\partial K} = \sum_{F\in \cF(K)} (h_K\alpha_F)^{-1}h_K\norm{\btau\cdot \n_K}^2_{0,F} \lesssim
 \sum_{F\in \cF(K)} h_K \norm{\btau}^2_{0,F}, 
 %\lesssim \norm{\btau}^2_{0, K},
\]
which means that 
\begin{equation}\label{T4}
 T_4 \lesssim\, \norm{(\btau, \hat{v})} \left( \norm{\tilde\pi_\psi}_{1/2,\Gamma}^2 +   
\sum_{K\in \cT_h} \norm{\alpha^{1/2} \pi_u}^2_{0,\partial K}  \right)^{1/2}.
\end{equation}
The result follows now directly from \eqref{T1}, \eqref{T2}, \eqref{T3} and \eqref{T4}.
\end{proof}
Let us introduce the errors  
\[
 e_{\bsigma} := \bsigma - \bsigma_h, \quad e_u := u - u_h,\quad  e_\psi := \psi - \psi_h \quad \text{and} \quad
 e_{\hat{u}} := (u - u_h, \psi - \psi_h).
\]
We notice that, under the regularity hypothesis of Proposition \ref{consistency0}, we have the following Galerkin orthogonality
\begin{equation}\label{consistency}
 A(e_{\bsigma},e_{\hat{u}}; \btau, \hat{v} ) = 0 \quad\forall (\btau, \hat{v})\in \bSigma_h \times (V_h\times \Psi_h).
\end{equation}
\begin{lemma}
 There exists $\underbar C>0$ such that 
\begin{equation}\label{coer}
\sup_{(\btau, \hat{v})\in \bSigma_h\times (V_h \times \Psi_h)} 
\frac{A(\bsigma, \hat{u};\btau, \hat{v})}{ \norm{(\btau, \hat{v})}_{\#} } \geq \underbar{C}\, \norm{(\bsigma, \hat{u})}
\end{equation}
for all $(\bsigma,\hat{u})\in  \bSigma_h\times (V_h \times \Psi_h)$.
\end{lemma}
\begin{proof}
It follows straightforwardly  from \eqref{eq-coercive-1} and \eqref{rn2} 
that, for all $(\btau, \hat{v}) \in \bSigma_h\times (V_h \times \Psi_h)$,
 \begin{multline*}
  A(\btau, \hat{v};\btau, \hat{v})\\ = 
  \norm{\btau}^2_{0,\Omega} + \dual{\btau\cdot \n, V (\btau\cdot \n)}_{\Gamma} + 
 \norm{\alpha^{1/2} \jump{\hat{v}}}^2_{0,\cF_h} + 
 \dual{V\curl_h \varphi, \curl_h\varphi }_\Gamma +  \norm{\nu^{1/2}\jump{\varphi}}^2_{0, \cE_h}
 \\[1ex]
 \gtrsim \left(\norm{\btau}^2_{0,\Omega} + \norm{\btau\cdot\n}^2_{-1/2,\Gamma} + 
  \norm{\alpha^{1/2} \jump{\hat{v}}}^2_{0,\cF_h} + 
 \norm{\curl_h \varphi}^2_{-1/2,\Gamma} + 
 \norm{\nu^{1/2}\jump{\varphi}}^2_{0, \cE_h}\right)
 \\[1ex] =  \norm{(\btau, \hat{v})}_{\#}^2 \gtrsim \norm{(\btau, \hat{v})} \norm{(\btau, \hat{v})}_{\#},
 \end{multline*}
 which proves the result. 
\end{proof}

\begin{theorem}\label{cea}
Under the hypothesis of Proposition \ref{consistency0},  
\[
 \norm{(e_{\bsigma}, e_{\hat{u}})} \leq (1 + \frac{\bar C}{\underbar C})|\log{h}|^{1/2} 
  \norm{(\pi_{\bsigma}, \pi_{\hat{u}})}_*.
\]
\end{theorem}
\begin{proof}
The Galerkin orthogonality \eqref{consistency} and \eqref{coer} yield 
\begin{multline*}
 \underbar{C}\, \norm{(\bsigma_h - \bPi \bsigma, (u_h - \Pi u, \psi_h - \tilde\pi \psi ))} \leq \\[1ex]
  \sup_{(\btau, \hat{v})\in \bSigma_h\times (V_h \times \Psi_h)} 
  \frac{A( \bsigma_h - \bPi \bsigma, (u_h - \Pi u, \psi_h - \tilde\pi \psi ) ; \btau, \hat{v} )}
  {\norm{(\btau, \hat{v})}_{\#}}=\\[1ex]
  \sup_{(\btau, \hat{v})\in \bSigma_h\times (V_h \times \Psi_h)} 
  \frac{A( \pi_{\bsigma}, \pi_{\hat{u}} ; \btau, \hat{v} )}
  {\norm{(\btau, \hat{v})}_{\#}}
\end{multline*}
Applying Lemma \ref{boundA} we deduce that 
\begin{equation}\label{interm}
 \norm{(\bsigma_h - \bPi \bsigma, (u_h - \Pi u, \psi_h - \tilde\pi \psi ))} \leq \frac{\bar C}{\underbar C}
 |\log{h}|^{1/2} 
 \norm{(\pi_{\bsigma}, \pi_{\hat{u}})}_*
\end{equation}
and the result follows from triangle inequality
\[
 \norm{(e_{\bsigma}, e_{\hat{u}})} \leq \norm{(\pi_{\bsigma}, \pi_{\hat{u}})} + 
 \norm{(\bsigma_h - \bPi \bsigma, (u_h - \Pi u, \psi_h - \tilde\pi \psi ))}.
\]
\end{proof}

\subsection{Asymptotic error estimates}
In this section we need to handle functions that are piecewise smooth on the boundary 
$\Gamma$ of the polyhedron $\Omega$. Let $\set{\Gamma_1, \cdots, \Gamma_N}$ be the open polygons, 
contained in different hyperplanes of $\mathbb{R}^3$, such that $\Gamma= \cup_{j= 1}^N \overline{\Gamma}_j$. 
For any $t\geq 0$, we consider the broken Sobolev space $H^t_{\pw}(\Gamma):= \prod_j H^t(\Gamma_j)$
endowed with the graph norm
\[
 \norm{\varphi}^2_{t,\pw,\Gamma}: = \sum_{j= 1}^N \norm{\varphi}^2_{H^t(\Gamma_j)}. 
\] 
Let us recall some well-known approximation properties related with the (local and global) projection and interpolation 
operators.
\begin{lemma}\label{v}
For all $K\in \cT_h$, if  $w\in H^{r+1}(K)$ with $r\geq 0$, then
 \[
  h_K\norm{\nabla (w - \Pi_K w) }_{0,K}  + \norm{w - \Pi_K w}_{0,K} \lesssim h_K^{\min\{r,m\}+1} \norm{w}_{r+1,K}
 \]
 and
\[
 \norm{w - \Pi_K w}_{0,\partial K} \lesssim h_K^{\min\{r,m\}+1/2} \norm{w}_{r+1,K}.
\]
\end{lemma}
\begin{lemma}\label{btau}
For all $K\in \cT_h$, if $\btau\in H^{r}(K)^3$ with $r> 1/2$, then
 \[
  \norm{\btau - \bPi_K \btau}_{0,K} \lesssim h_K^{\min\{r,m\}} \norm{\btau}_{r,K},
 \]
 and
 \[
  \norm{(\btau - \bPi_K \btau)\cdot \n_K}_{0,\partial K} \lesssim h_K^{\min\{r,m\}-1/2} \norm{\btau}_{r,K}.
 \]
\end{lemma}
\begin{proof}The first estimate is standard (cf. \cite{rt91}), we only prove the second one. 
 Let us denote by $\hat K$ the reference tetrahedron   and consider the Piola transformation
 \[
  \btau = \frac{1}{\textrm{det}(B_K)}B_K\hat{\btau}
 \]
 where $B_K$ is the matrix associated with the affine map from $\hat K$ onto $K$. We consider a face 
 $F$ of $T$ and we denote by $\hat F$ the corresponding face in $\hat T$ under the affine map. It is easy to show that 
 \[
  \abs{F} \norm{\btau \cdot \n_F}^2_{0,F} = \abs{\hat F} \norm{\hat \btau \cdot \n_{\hat F} }^2_{0, \hat F}
  \quad \forall \btau\in \mathbf{H}^1(K).
 \]
Let us denote by $\hat \bPi$ the Raviart-Thomas interpolation in $\mathbf{RT}_m(\hat K)$. 
It follows from the trace theorem in $\mathbf{H}^{\min\{r,m\}}(\hat K)$ and the Bramble-Hilbert theorem that 
\[
 \norm{(\hat\btau - \hat{\bPi} \hat{\btau})\cdot \n_{\hat F} }_{0,\hat{F}} \lesssim  |\hat \btau |_{\min\{r,m\},\hat K},
\]
where $|\cdot|_{\min\{r,m\},\hat K}$ stands for the semi-norm in $\mathbf{H}^{\min\{r,m\}}(\hat K)$.
Combining the last two estimates and recalling  that $\widehat{\bPi_K \btau} = \hat\bPi \hat{\btau}$
we obtain
\[
 \norm{\btau - \bPi_K \btau}_{0,F} \lesssim \abs{F}^{-1/2} |\hat \btau |_{\min\{r,m\},\hat K},
\]
and transforming back to $K$ this gives
\[
 \norm{\btau - \bPi_K \btau}_{0,F} \lesssim \abs{\hat F}^{1/2} \abs{F}^{-1/2} \abs{\textrm{det}(B_K)}^{1/2} \norm{B_K^{-1}}
 \norm{B_K}^{\min\{r,m\}} \norm{\btau }_{r, K}.
\]
The result follows now from the shape regularity of the partition $\cT_h$.
\end{proof}

% Let $\set{\Gamma_1, \cdots, \Gamma_N}$ denote the faces of $\Gamma$. For $t\geq 0$ we consider the space  
% $H^t(\Gamma_j)$ of functions that can be extended to functions in $H^t(\mathbb{R}^2)$  
% after identification of $\Gamma_j$ with a subset of $\mathbb{R}^2$. This space is endowed with the 
% image topology of the restriction operator. Finally, we set $H^t_{\pw}}(\Gamma):= \prod_j H^t(\Gamma_j)$ 
% and denote its norm by $\norm{\cdot}_{t,\pw,\Gamma}$.
\begin{lemma}\label{varphiHmedio}
Assume that $\varphi\in H_{\pw}^{r+1/2}(\Gamma)\cap H^1(\Gamma)$ with $r>1/2$, then
\[
 \norm{\varphi - \tilde\pi_T \varphi}_{0,T} \lesssim h_T^{\min\{r+1/2,m+1\}} \norm{\varphi}_{r+1/2,T} \quad \forall T\in \cG_h
\]
and
 \[
  \norm{\varphi - \tilde\pi \varphi}_{t,\Gamma} \lesssim h^{\min\{r+1/2,m+1\}-t} 
  \norm{\varphi}_{r+1/2,\pw,\Gamma},\quad t\in \{0,1/2, 1\}.
 \]
\end{lemma}
\begin{proof}
 See \cite[Proposition 4.1.50]{sauterSchwab}
\end{proof}

\begin{lemma}\label{sigman}
Assume that $\lambda\in H^{-1/2}(\Gamma)\cap H_{\pw}^r(\Gamma)$ for some $r \geq 0$ and let 
\[
 \Lambda_h = \set{\mu\in L^2(\Gamma); \,\, \mu|_{T} \in \cP_{m-1}(T)
 \quad \forall T\in \cG_h}.
\]
Then,
\[
 \norm{\lambda - \mu_h}_{-t,\Gamma}\lesssim h^{\min\{ r,m \}+t} \norm{\lambda}_{r,\pw,\Gamma}\quad t\in \{0, 1/2\},
\]
where $\mu_h$ the best $L^2(\Gamma)$ approximation of $\lambda$ in $\Lambda_h$. 
\end{lemma}
\begin{proof}
 See \cite[Theorem 4.3.20]{sauterSchwab}.
\end{proof}

\begin{theorem}\label{main}
Assume that  that the solution of \eqref{ModelProblem} satisfies 
$u\in H^{s+2}(\Omega)$, $\psi\in H_{\pw}^{s+3/2}(\Gamma)\linebreak\cap H^1(\Gamma)$ and 
$\bsigma \cdot \n \in H_{\pw}^{s+1}(\Gamma)$ for some 
$s\geq 0$. Then, 
\begin{multline*} 
 \norm{(\bsigma - \bsigma_h, (u - u_h, \psi - \psi_h))} \lesssim |\log{h}|^{1/2}\, h^{\min\{s+1, m\}} \\[1ex]
 (\norm{u}_{s+2,\Omega} + \norm{\psi}_{s+3/2,\pw,\Gamma} + \norm{\bsigma \cdot \n}_{s+1,\pw,\Gamma} ).
\end{multline*}
\end{theorem}
\begin{proof}
Let us first notice that, thanks to \eqref{localUnif}, \eqref{coef} and Proposition \ref{card}, 
\[
 \norm{\alpha^{1/2} \jump{\hat{v}}}^2_{0,\cF_h} \lesssim \left( \sum_{K\in \cT_h} h_K^{-1}\norm{v}^2_{0,\partial K}
 + \sum_{T\in \cG_h} h_T^{-1} \norm{\varphi}^2_{0,T}\right)\quad \forall (v,\varphi)\in H^1(\cT_H)\times H^{1/2}(\Gamma).
\]
It follows that 
\begin{multline}\label{interp}
 \norm{(\pi_{\bsigma}, \pi_{\hat{u}})}_* \lesssim \Big( 
 \norm{\pi_{\bsigma}}^2_{0,\Omega} +  \sum_{K\in \cT_h}h_K\norm{ \pi_{\bsigma}\cdot \n_K}^2_{0,\partial K}
 + \norm{\pi_{\bsigma}\cdot\n}^2_{-1/2,\Gamma}+
 \norm{\pi_{\bsigma}\cdot\n}^2_{0,\Gamma}+
 \\[1ex]
 \sum_{K\in \cT_h} h_K^{-1}\norm{ \pi_u}^2_{0,\partial K}
 + h^{-1} \norm{\tilde\pi_\psi}^2_{0,\Gamma} + 
 \norm{\tilde\pi_\psi}^2_{1/2,\Gamma} + h \abs{\tilde\pi_\psi}^2_{1,\Gamma}
 \Big)^{1/2}.
\end{multline}
Using Lemma \ref{btau} we obtain
\begin{equation}\label{est1}
\norm{\pi_{\bsigma}}^2_{0,\Omega} +  \sum_{K\in \cT_h}h_K\norm{ \pi_{\bsigma}\cdot \n_K}^2_{0,\partial K} \leq 
C_2 h^{2\min\{s+1,m\}} \norm{u}^2_{s+2,\Omega}.
\end{equation}
On the other hand,  we recall that (by definition of the Raviart-Thomas 
interpolation operator) $(\bPi \bsigma)|_\Gamma \cdot \n$ is the best $L^2(\Gamma)$ approximation of 
$\bsigma \cdot \n$ in $\Lambda_h$. 
Consequently, by virtue of 
the regularity assumption  
$\bsigma\cdot \n\in H^{1+s}_{\pw}(\Gamma)$ and 
Lemma \ref{sigman},
\begin{equation}\label{est2}
 \norm{\pi_{\bsigma}\cdot \n}_{-1/2,\Gamma}\lesssim\, h^{\min\{ s+3/2,m +1/2 \}} \norm{\bsigma\cdot \n}_{1+s,\pw,\Gamma}
\end{equation}
and
\begin{equation}\label{est2b}
% (\norm{\pi_{\bsigma}\cdot \n}_{-1/2,\Gamma}\leq) \quad 
 \norm{\pi_{\bsigma}\cdot \n}_{0,\Gamma}
 \lesssim\, h^{\min\{ s+1, m \}} \norm{\bsigma\cdot \n}_{1+s,\pw,\Gamma}
\end{equation}
Applying now the estimates given in Lemma \ref{varphiHmedio} we deduce that
\begin{equation}\label{est3}
\sum_{T\in \cG_h} h_T^{-1} \norm{\tilde\pi_\psi}^2_{0,T} + 
 \norm{\tilde\pi_\psi}^2_{1/2,\Gamma}
 + h^{-1}\norm{\tilde\pi_\psi}^2_{0,\Gamma} + h\norm{\tilde\pi_\psi}^2_{1,\Gamma}
 \lesssim\,  h^{2\min\{s+1,m+1/2\}} \norm{\psi}^2_{s+3/2,\pw,\Gamma}.
 \end{equation}
 Finally, Lemma \ref{v} proves that 
 \begin{equation}\label{est4}
 \sum_{K\in \cT_h} h_K^{-1}\norm{ \pi_u}^2_{0,\partial K} \lesssim\, h^{2\min\{s+1,m\}} \norm{u}^2_{s+2,\Omega}.
 \end{equation}
  Plugging \eqref{est1}, \eqref{est2b}, \eqref{est3} and \eqref{est4} in \eqref{interp} and using that 
  $\min\{s+1, m+ 1/2 \}  \geq \min\{s+1, m\}$ yield
 \begin{equation}\label{globest}
  \norm{(\pi_{\bsigma}, \pi_{\hat{u}})}_* \lesssim\, h^{\min\{s+1, m\}} \left(\norm{u}^2_{s+2,\Omega} + 
  \norm{\psi}^2_{s+3/2,\pw,\Gamma} + \norm{\bsigma\cdot \n}^2_{1+s,\pw,\Gamma}\right)^{1/2}.
 \end{equation}
It follows now from Lemma \ref{cea} that
 \[
  \norm{(e_{\bsigma}, e_{\hat{u}})} \lesssim |\log h|^{1/2}\, h^{\min\{s+1, m\}} \left(\norm{u}^2_{s+2,\Omega} + 
  \norm{\psi}^2_{s+3/2,\pw,\Gamma} + \norm{\bsigma\cdot \n}^2_{1+s,\pw,\Gamma}\right)^{1/2}.
 \]
\end{proof}

\begin{remark}
 We point out that if $g_0 = g_1 = 0$ and the solution is harmonic in a neighborhood of $\Gamma$ 
 then, the boundary regularity assumption $\psi\in H_{\pw}^{s+3/2}(\Gamma)\cap H^1(\Gamma)$ and 
 $\bsigma \cdot \n \in H_{\pw}^{s+1}(\Gamma)$ will hold true for any $s\geq 0$. This condition 
 can always be fulfilled be choosing $\Omega$ big enough to contain the jumps of the solution. 
\end{remark}

\section{Conforming approximation on the boundary}\label{s5}
With little more effort we can provide the convergence analysis for a Galerkin scheme based on a conforming 
BEM-approximation. To this end, we introduce 
\begin{align*}
\widetilde  \Psi_h    &:= \prod_{T\in \mathcal{G}_h} \cP_m(T) \cap H^{1/2}_0(\Gamma)
 \subset H_0^{1/2}(\Gamma):=H^{1/2}(\Gamma)\cap L^2_0(\Gamma) 
\end{align*}
and consider the problem: find $(\bsigma_h, \hat{u}_h)\in \bSigma_h\times (V_h \times \widetilde \Psi_h)$ such that
\begin{equation}\label{ldg-FemConfBem}
 \begin{array}{rcll}
 a(\bsigma_h,\btau) + b(\btau, \hat{u}_h) &=& \dual{g_0, \btau\cdot \n}_{\Gamma} + 
 \dual{Vg_1, \btau \cdot \n}_{\Gamma} &\forall \btau\in \bSigma_h\\[2ex]
 -b(\bsigma_h, \hat{v})+ c(\hat{u}_h, \hat{v}) &=& (f,v)_{\Omega} + \dual{\alpha g_0, v-\varphi}_{\Gamma}
 +\dual{( \frac{\text{id}}{2} + K')g_1, \varphi}_{\Gamma}
 & \forall \hat{v}\in V_h\times \widetilde\Psi_h.
 \end{array}
\end{equation}
Note that the restriction of the bilinear form $d(\cdot, \cdot)$, used in the definition of $c(\cdot, \cdot)$ and 
introduced in \eqref{d}, reduces to 
\[
 d(\psi, \varphi) := \dual{V\curl_\Gamma \psi, \curl_\Gamma\varphi }_\Gamma
\]
for functions $\psi$ and $\varphi$ in $\widetilde\Psi_h$. This will simplify considerably the analysis of the scheme.
All the other bilinear forms in \eqref{ldg-FemConfBem} remain unchanged. 
Apart from the fact that the shape regular conforming triangulation $\cG_h:=\set{T}$ is no longer needed 
to be quasi-uniform, in the sequel, we will use the same hypothesis on the triangulations 
and we will also use the same notations introduced in previous sections.  

The well-posedness and the consistency of the scheme \eqref{ldg-FemConfBem} follow 
by straightforward simplifications of the 
arguments used in the proofs of Propositions \ref{wellposed} and \ref{consistency0}.
\begin{prop}\label{consistencyConf}
The LDG-FEM/BEM  defined by \eqref{ldg-FemConfBem} provides a unique approximate solution $(\bsigma_h, (u_h,\psi_h))
\in \bSigma_h\times (V_h\times \widetilde\Psi_h)$. Moreover, if $u$ is the solution of \eqref{ModelProblem} in $\Omega$, 
$\bsigma:= \nabla u$, $\psi:=u|_\Gamma - g_0$ and  $u\in H^2(\Omega)$ then,  
 \[
  A(\bsigma, (u, \psi); \btau, (v,\varphi)) = F(\btau, (v, \varphi))\quad \forall \btau \in \bSigma_h, \quad 
  \forall (v,\varphi)\in V_h\times \widetilde\Psi_h.
 \]
\end{prop}
 
Reexamining carefully the proof of Lemma \ref{boundA} we obtain the following stability property for  
scheme \eqref{ldg-FemConfBem}.
\begin{lemma}\label{boundAconf}
Let us assume that $\bsigma \in H^{1/2+\varepsilon}(\Omega)^3$ with $\varepsilon> 0$.
 Then, 
 \[
  \abs{A(\pi_{\bsigma}, \pi_{\hat{u}}; \btau, \hat{v})} \lesssim \,  
  \norm{(\pi_{\bsigma}, \pi_{\hat{u}})}_*^c\,  \norm{(\btau, \hat{v})}^c
  \quad \forall  (\btau, \hat{v})\in \bSigma_h \times (V_h\times \widetilde\Psi_h),
 \]
 where 
 \[
 \norm{(\btau, \hat{v})}^c := \left(\norm{\btau}^2_{0,\Omega} + \norm{\btau\cdot\n}^2_{-1/2,\Gamma} + 
 \norm{\varphi}^2_{1/2,\Gamma} + \norm{\alpha^{1/2} \jump{\hat{v}}}^2_{0,\cF_h}\right)^{1/2}
\]
and 
\[
 \norm{(\btau, \hat{v})}^c_*:= \left(\norm{(\btau, \hat{v})}^2 + 
 \sum_{K\in \cT_h}\norm{\alpha^{-1/2} \btau\cdot \n_K}^2_{0,\partial K}+
\sum_{K\in \cT_h} \norm{\alpha^{1/2} v}^2_{0,\partial K}
 \right)^{1/2}.
\]
\end{lemma}

It is straightforward to deduce from \eqref{eq-coercive-1} and \eqref{eq-coercive-2} that
\begin{equation}\label{coerConf}
 A(\btau, \hat{v};\btau, \hat{v}) \gtrsim (\norm{(\btau, \hat{v})}^c)^2 
 \quad \forall \btau\in \mathbf{H}^1(\cT_h), \quad \forall \hat{v}=(v,\varphi) \in H^1(\cT_h)\times H^{1/2}_0(\Gamma).
\end{equation}
Combining \eqref{coerConf} with Lemma \ref{boundAconf} yields the following error estimate. 
\begin{theorem}\label{ceaConf}
Under the hypothesis of Proposition \ref{consistencyConf},  
\[
 \norm{(e_{\bsigma}, e_{\hat{u}})}^c \lesssim 
  \norm{(\pi_{\bsigma}, \pi_{\hat{u}})}^c_*.
\]
\end{theorem}

In the conforming BEM case, we can also provide an estimate of the error $u - u_h$ in 
the $L^2(\Omega)$-norm. To this end, we follow \cite{CastilloCockburn} and use a duality 
argument. For any $\rho \in L^2(\Omega)$ we consider the exterior problem 
\begin{equation*}\label{auxiliary}
\begin{array}{rcll}
\Delta w &=& \tilde\rho         &\text{in $\mathbb{R}^3$},\\[2ex]
        w &=& O(\disp\frac{1}{\abs{\x}})  &\text{as $\abs{\x}\to \infty$},
\end{array}
\end{equation*}
where $\tilde \rho$ is the extension by zero of $\rho$  outside $\Omega$. It is well known from the theory 
of regularity of elliptic problems that $w\in H^2(\Omega)$ and there exists $C_{reg}>0$ such that 
\begin{equation}\label{reg}
 \norm{w}_{2,\Omega}\leq C_{reg} \norm{\rho}_{0,\Omega}.
\end{equation}

\begin{prop}\label{ident}
 Let us introduce $\btau_\rho := \nabla w$ and $\hat{v}_\rho = (v_\rho, \varphi_\rho):=(-w|_\Omega,  -w|_\Gamma)$. Then,
 \[
  A(\bsigma,\hat{u}; \btau_\rho, \hat{v}_\rho ) = (u,\rho )_\Omega \quad \forall \bsigma\in \mathbf{H}^1(\cT_h), 
  \quad \forall \hat{u}:= (u,\psi)\in H^1(\cT_h)\times H^{1/2}(\Gamma). 
 \]
\end{prop}
\begin{proof}
 The definition \eqref{A} of the bilinear form $A(\cdot, \cdot)$ yields
 \begin{multline*}
  A(\bsigma,\hat{u}; \btau_\rho, \hat{v}_\rho) = ( \bsigma, \nabla w )_{\cT_h} + \dual{\bsigma\cdot\n, V\nder{w}}_\Gamma - 
  (\nabla_h u, \nabla w)_{\cT_h} +\dual{\nder{w}, (\frac{\text{id}}{2} - K)\psi}_\Gamma\\[1ex]
  + \dual{\jump{u}, \mean{\nabla w}}_{\cF_h^0} + \dual{\nder{w}, u-\psi}_{\cF_h^\partial} - 
  (\nabla w, \bsigma)_{\cT_h} + \dual{\bsigma\cdot \n, (\frac{\text{id}}{2} - K)w}_\Gamma
  - \dual{W \psi, w}_\Gamma.
 \end{multline*}
Taking into account that 
 \[
  V\nder{w} + (\frac{\text{id}}{2} - K)w = 0 \quad \text{and} \quad (\frac{\text{id}}{2} + K')\nder{w} + Ww = 0
 \]
we deduce that  
\begin{multline*}
 A(\bsigma,\hat{u}; \btau_\rho, \hat{v}_\rho) = - \sum_{K\in \cT_h} (\nabla u, \nabla w)_K + \dual{\jump{u}, \mean{\nabla w}}_{\cF_h^0} 
 + \dual{\nder{w}, u}_{\cF_h^\partial}\\[1ex]
 = \sum_{K\in \cT_h} (\Delta w, u)_K - \sum_{K\in \cT_h} \dual{\nder{w}, u}_{\partial K}+ \dual{\jump{u}, \mean{\nabla w}}_{\cF_h^0} 
 + \dual{\nder{w}, u}_{\cF_h^\partial} = (\rho, u)_\Omega 
\end{multline*}
and the result follows.
\end{proof}

\begin{lemma}\label{duality}
 It holds  that 
 \[
  \norm{e_u}_{0,\Omega} \lesssim
 \norm{(\pi_{\bsigma}, \pi_{\hat{u}})}_{**}
 \sup_{\rho\in L^2(\Omega)}
 \frac{\norm{(\pi_{\btau_\rho}, \pi_{\hat{v}_\rho})}_{**}}{\norm{\rho}_{0,\Omega}},
 \]
 where
 \[
 \norm{(\btau, \hat{v})}_{**} := \left( 
 (\norm{(\btau, \hat{v})}^c_*)^2 + \sum_{K\in \cT_h}\norm{\nabla v_K}^2_{0,K}
 \right)^{1/2}.
\]
\end{lemma}
\begin{proof}
We deduce from Proposition \ref{ident} that 
\[
 A(e_{\bsigma},e_{\hat{u}}; \btau_\rho, \hat{v}_\rho ) = (e_u,\rho )_\Omega,
\]
and it follows from the definition of $A(\cdot, \cdot)$ that  
\begin{multline*}
 (e_u,\rho )_\Omega = A( e_{\bsigma}, e_{\hat{u}} ; \pi_{\btau_\rho}, \pi_{\hat{v}_\rho} ) = 
 A( \bPi e_{\bsigma}, (\Pi e_u, \tilde\pi e_\psi ) ; \pi_{\btau_\rho}, \pi_{\hat{v}_\rho} ) + 
 A( \pi_{\bsigma}, \pi_{\hat{u}} ; \pi_{\btau_\rho}, \pi_{\hat{v}_\rho} )\\[1ex]
 = A(-\pi_{\btau_\rho} , \pi_{\hat{v}_\rho} ; -\bPi e_{\bsigma}, (\Pi e_u, \tilde\pi e_\psi )  ) + 
 A( \pi_{\bsigma}, \pi_{\hat{u}} ; \pi_{\btau_\rho}, \pi_{\hat{v}_\rho} ).
\end{multline*}
Consequently, by virtue of Lemma \ref{boundAconf} and Theorem \ref{ceaConf},% and \eqref{interm}, 
\begin{multline*}
 \norm{e_u}_{0,\Omega} \leq \sup_{\rho\in L^2(\Omega)}
 \frac{\abs{ A(-\pi_{{\btau}_\rho} , \pi_{\hat{v}_\rho} ; -\bPi e_{\bsigma}, (\Pi e_u, \tilde\pi e_\psi )  ) }}{\norm{\rho}_{0,\Omega}} + 
 \sup_{\rho\in L^2(\Omega)}\frac{\abs{ A( \pi_{\bsigma}, \pi_{\hat{u}} ; \pi_{\btau_\rho}, \pi_{\hat{v}_\rho} ) }}{\norm{\rho}_{0,\Omega}}\\[1ex]
 \lesssim
 \norm{(\pi_{\bsigma}, \pi_{\hat{u}})}_*^c
 \sup_{\rho\in L^2(\Omega)}
 \frac{\norm{(\pi_{\btau_\rho}, \pi_{\hat{v}_\rho})}_*^c}{\norm{\rho}_{0,\Omega}}+ 
 \sup_{\rho\in L^2(\Omega)}\frac{\abs{ A( \pi_{\bsigma}, \pi_{\hat{u}} ; \pi_{\btau_\rho}, \pi_{\hat{v}_\rho} ) }}{\norm{\rho}_{0,\Omega}}.
 \end{multline*}
Moreover, it follows from the straightforward  estimate 
\[
 \abs{ A( \pi_{\bsigma}, \pi_{\hat{u}} ; \pi_{\btau_\rho}, \pi_{\hat{v}_\rho} ) } \lesssim 
 \norm{(\pi_{\bsigma}, \pi_{\hat{u}})}_{**} \norm{(\pi_{\btau_\rho}, \pi_{\hat{v}_\rho})}_{**},
\]
that
\begin{equation*}\label{duality1}
 \norm{e_u}_{0,\Omega} \lesssim
 \norm{(\pi_{\bsigma}, \pi_{\hat{u}})}_{**}
 \sup_{\rho\in L^2(\Omega)}
 \frac{\norm{(\pi_{\btau_\rho}, \pi_{\hat{v}_\rho})}_{**}}{\norm{\rho}_{0,\Omega}},
\end{equation*}
which proves the result. 
\end{proof}
\begin{theorem}\label{mainConf}
Assume that the solution of \eqref{ModelProblem} satisfies 
$u\in H^{s+2}(\Omega)$ and $\psi\in H_{\pw}^{s+3/2}(\Gamma)\cap H^1(\Gamma)$ for some 
$s\geq 0$. Then, 
\[
 \norm{u - u_h}_{0,\Omega} + h \norm{(\bsigma - \bsigma_h, (u - u_h, \psi - \psi_h))}^c \lesssim \, h^{\min\{s+2, m+1\}	}
 (\norm{u}_{s+2,\Omega} + \norm{\psi}_{s+3/2,\pw,\Gamma} ).
\]
\end{theorem}
\begin{proof}
 Similar arguments to those used in the proof of  Theorem \ref{main} yield
 \[
  \norm{(e_{\bsigma}, e_{\hat{u}})}^c \lesssim\, h^{\min\{s+1, m\}} \left(\norm{u}^2_{s+2,\Omega} + 
  \norm{\psi}^2_{s+3/2,\pw,\Gamma} \right)^{1/2}.
 \]
 By virtue of \eqref{globest} and Lemma \ref{v}, it is also true that 
 \[
  \norm{(\pi_{\bsigma}, \pi_{\hat{u}})}_{**} \lesssim\, h^{\min\{s+1, m\}} \left(\norm{u}^2_{s+2,\Omega} + 
  \norm{\psi}^2_{s+3/2,\pw,\Gamma} \right)^{1/2}
 \]
 and 
 \[
  \norm{(\pi_{\btau_\rho}, \pi_{\hat{v}_\rho})}_{**}  \lesssim\, h \norm{w}_{2,\Omega} \lesssim \, h \norm{\rho}_{0,\Omega},
 \]
where the last estimate comes from \eqref{reg}.
Consequently, using Lemma \ref{duality} we conclude that 
\[
 \norm{u-u_h}_{0,\Omega} \lesssim  h^{\min\{s+1, m\}+1} \left(\norm{u}^2_{s+2,\Omega} + 
  \norm{\psi}^2_{s+3/2,\pw,\Gamma} \right)^{1/2}
\]
and the result follows.
\end{proof}

\section{Numerical results}\label{s6}
In this section we present a numerical experiment confirming the theoretical error estimate
obtained for the LDG-FEM/DG-BEM scheme \eqref{ldg-FemBem} and the LDG-FEM/BEM scheme \eqref{ldg-FemConfBem}. 
For simplicity we consider our model problem in two dimensions. 
The corresponding theory and results from three dimensions apply with trivial modifications.

We choose $\Omega=(0,1)^2$ and select the data so that the exact solution is given by 
\[
   u(x_1, x_2)=\sin(10 x_1+3 x_2)\quad \text{in $\Omega$}\quad \text{and}\quad 
   u^e(x_1, x_2)=\frac{x_1+x_2-1}{(x_1-0.5)^2+(x_2-0.5)^2}\quad \text{in $\Omega_e$}.
\]
\begin{figure}
\begin{center}
\includegraphics[width=0.75\textwidth]{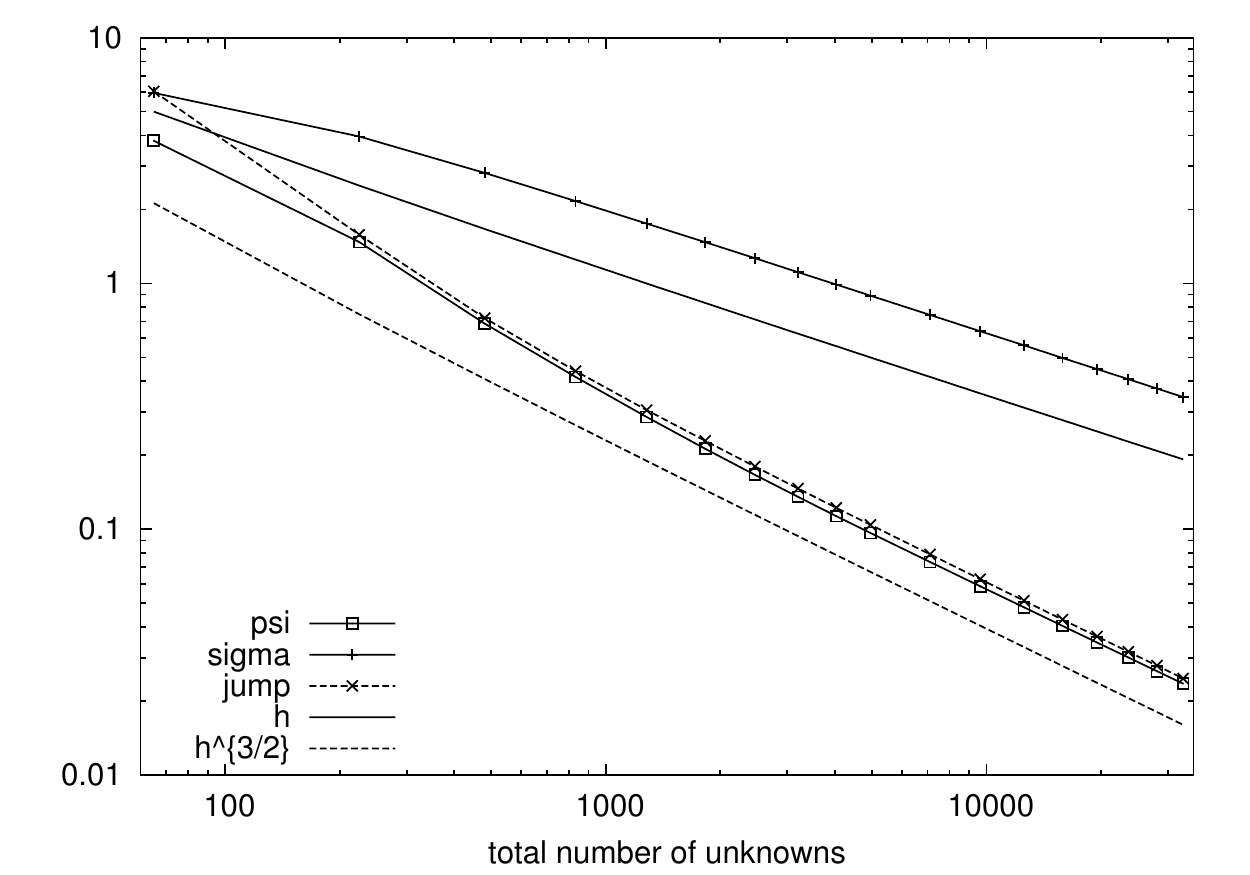} 
\end{center}
\caption{Errors and $O(h)$, $O(h^{3/2})$ versus total number of unknowns for the LDG-FEM/DG-BEM method.}
\label{fig}
\end{figure}
\begin{figure}
\begin{center}
\includegraphics[width=0.75\textwidth]{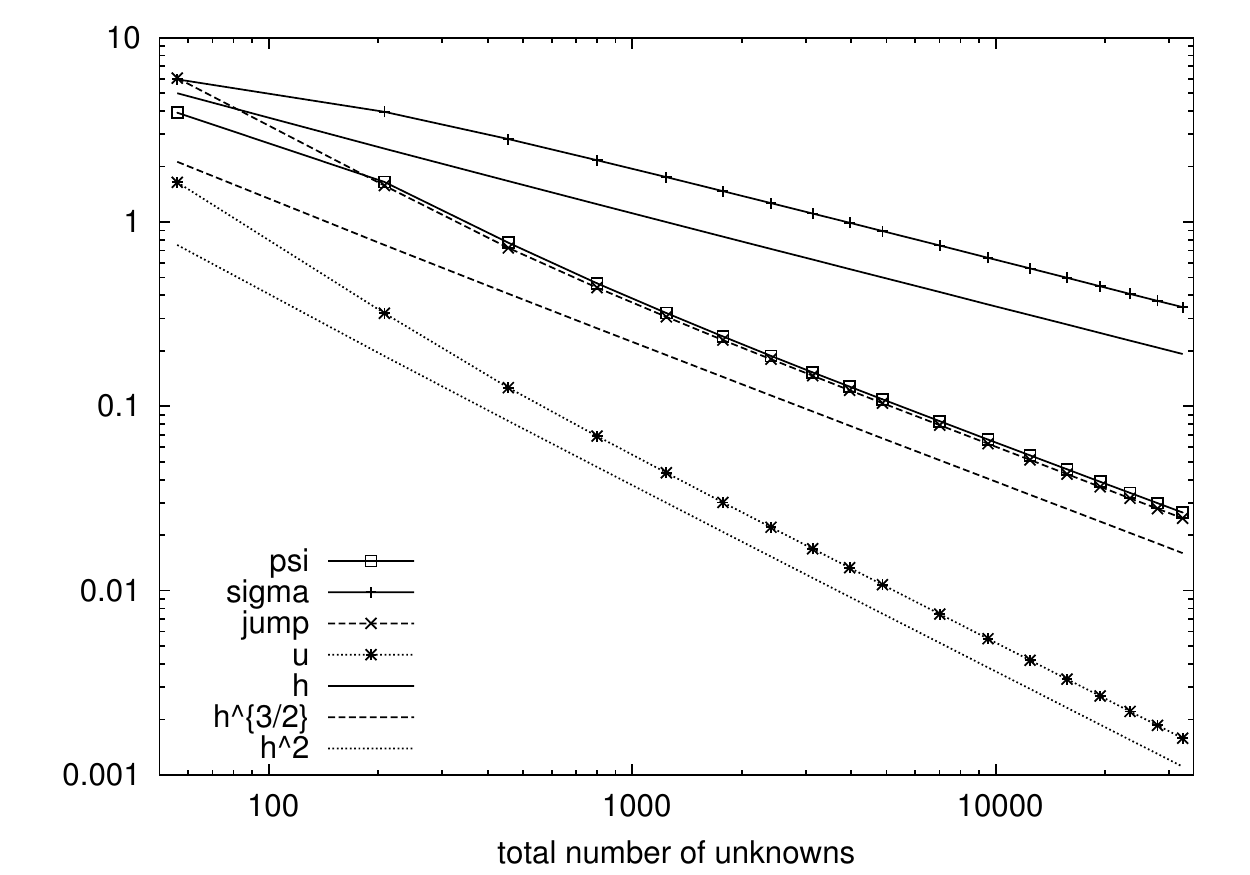} 
\end{center}
\caption{Errors and $O(h)$, $O(h^{3/2})$, $O(h^{2})$ versus total number of unknowns for the LDG-FEM/BEM method.}
\label{fig0}
\end{figure}
We consider uniform triangular meshes $\cT_h$ on $\Omega$ and inherited meshes $\cG_h$ on $\Gamma$
(and for simplicity denote $h$ to be the length of the shortest edge).
Lowest order discrete spaces are taken, i.e. $m=1$, so that
\[
 V_h = \prod_{K\in \mathcal{T}_h} \cP_1(K), \quad
 \bSigma_h = \prod_{K\in \cT_h} \mathbf{RT}_1(K),
\]
\[
\Psi_h= \set{\varphi \in L_0^{2}(\Gamma); \quad \varphi|_T\in P_1(T)\quad \forall T\in \cG_h}
\]
and
\[
\tilde\Psi_h= \set{\varphi \in H_0^{1/2}(\Gamma); \quad \varphi|_T\in P_1(T)\quad \forall T\in \cG_h}.
\]
Moreover, we select $\bbeta$ to be normal on the interior edges (in a certain direction)
with $|\bbeta|=1$, and $\alpha=h_{\cF}^{-1}$. In this case Theorem \ref{main} proves the  
behaviors $\norm{\psi-\psi_h}_{1/2,\cG_h}= O(h)$,   
$\norm{\bsigma-\bsigma_h}_{0,\Omega} + \norm{(\bsigma-\bsigma_h)\cdot\n}_{-1/2,\Gamma}=O(h)$ and 
$\norm{\jump{\hat u_h}}_{0,\cF_h} = O(h^{3/2})$.
% \begin{align*}
%  &\norm{\psi-\psi_h}_{1/2,\cG_h} = O(h),\\
%  &\norm{\bsigma-\bsigma_h}_{0,\Omega} + \norm{(\bsigma-\bsigma_h)\cdot\n}_{-1/2,\Gamma} = O(h),\\
%  &\norm{\jump{\hat u_h}}_{0,\cF_h} = O(h^{3/2}),
%  %\\&\norm{u - u_h}_{0,\Omega} =O(h^2).
% \end{align*}
In Figure~\ref{fig} the errors 
$\norm{\psi-\psi_h}_{[0,1],1/2,\cG_h}$, $\norm{\bsigma-\bsigma_h}_{0,\Omega}$ and 
$\norm{\jump{\hat u_h}}_{0,\cF_h}$ are labeled  \text{``psi''}, \text{``sigma''}
and \text{``jump''} respectively and they are depicted versus the total number of unknowns 
on a double-logarithmic scale.
% \begin{align*}
%    &\norm{\psi-\psi_h}_{[0,1],1/2,\cG_h} &&\hspace{-7em}\text{(``psi'')},\\
%    &\norm{\bsigma-\bsigma_h}_{0,\Omega}    &&\hspace{-7em}\text{(``sigma'')},\\
%    &\norm{\jump{\hat u_h}}_{0,\cF_h}       &&\hspace{-7em}\text{(``jump'')}%\\
%    %&\norm{u - u_h}_{0,\Omega}              &&\hspace{-7em}\text{(``u'')}
% \end{align*}
Here,
\[
   \norm{\psi-\psi_h}_{[0,1],1/2,\cG_h}
   :=
   \Bigl(\norm{\psi-\psi_h}_{0,\Gamma}^2
         + \sum_{T\in\cG_h} \|\psi-\psi_h\|_{0,T} |\psi-\psi_h|_{1,T}\Bigr)^{1/2}
\]
which, by interpolation, is an upper bound for $\norm{\psi - \psi_h}_{1/2,\cG_h}$
up to a constant factor. 
The curves $h$ and $h^{3/2}$ are also given multiplied by appropriate factors to
shift them closer to the corresponding curves. The numerical experiment confirms the convergence rates
$\norm{\bsigma-\bsigma_h}_{0,\Omega}=O(h)$, $\norm{\jump{\hat u_h}}_{0,\cF_h}=O(h^{3/2})$ and 
suggests the stronger convergence
$\norm{\psi-\psi_h}_{1/2,\cG_h}=O(h^{3/2})$.

In Figure~\ref{fig0}, the errors 
$\norm{\psi-\psi_h}_{[0,1],1/2,\Gamma}$, $\norm{\bsigma-\bsigma_h}_{0,\Omega}$, $\norm{u-u_h}_{0,\Omega}$ and  
$\norm{\jump{\hat u_h}}_{0,\cF_h}$ are labeled  \text{``psi''}, \text{``sigma''}, \text{``u''}
and \text{``jump''} respectively and they are  represented  again versus the total number 
of unknowns on a double-logarithmic scale.
% \begin{align*}
%    &\norm{\psi-\psi_h}_{[0,1],1/2,\cG_h} &&\hspace{-7em}\text{(``psi'')},\\
%    &\norm{\bsigma-\bsigma_h}_{0,\Omega}    &&\hspace{-7em}\text{(``sigma'')},\\
%    &\norm{\jump{\hat u_h}}_{0,\cF_h}       &&\hspace{-7em}\text{(``jump'')}%\\
%    %&\norm{u - u_h}_{0,\Omega}              &&\hspace{-7em}\text{(``u'')}
% \end{align*}
Here,
\[
   \norm{\psi-\psi_h}_{[0,1],1/2,\Gamma}
   := \Bigl(\norm{\psi-\psi_h}_{0,\Gamma}^2 + \|\psi-\psi_h\|_{0,\Gamma} |\psi-\psi_h|_{1,\Gamma}\Bigr)^{1/2}
\]
which, by interpolation, is an upper bound for $\norm{\psi - \psi_h}_{1/2,\Gamma}$
up to a constant factor.  The numerical results are in agreement with the convergence rates
$\norm{\bsigma-\bsigma_h}_{0,\Omega}=O(h)$, $\norm{u - u_h}_{0,\Omega}=O(h^2)$ and 
$\norm{\jump{\hat u_h}}_{0,\cF_h}=O(h^{3/2})$ obtained in Theorem \ref{mainConf}, and
indicate the stronger convergence $\norm{\psi-\psi_h}_{1/2,\Gamma}=O(h^{3/2})$.

%%%%%%%%%%%%%%%%%%%%%%%%%%%%%%%%%%%%%%%%%%%%%%%%%%%%%%%%%%%%%%%%%%%%%%%%%%%

\end{document}